 \documentclass[11pt,a4paper,reqno]{amsart}
\allowdisplaybreaks
\usepackage{amsmath,float}
\date{}
 \usepackage{amssymb,latexsym, esint}
 \usepackage{amsthm}
 \usepackage[latin1]{inputenc}
 \usepackage{enumerate}
\usepackage[dvips]{graphicx}
\usepackage{color}
\usepackage{verbatim}
\DeclareGraphicsExtensions{.pdf,.png,.jpg}
\def\supp{\mathop{\rm supp}}
\newcommand{\dist}{{\rm dist}}
\newcommand{\R}{\mathbb R}

\newtheorem{thm}{Theorem}[section]
\newtheorem{lemma}{Lemma}[section]
 \newtheorem{prop}{Proposition}[section]
\newtheorem{cor}{Corollary}[section]

\theoremstyle{definition}
\newtheorem{defi}{Definition}[section]

\newtheorem{rem}{Remark}[section]

\begin{document}

\title[Localization]
{Localization and landscape functions on quantum graphs}

\author[E. M. Harrell]{Evans M. Harrell II}
\address{School of Mathematics,
Georgia Institute of Technology,
Atlanta GA 30332-0160, USA}
\email{harrell@math.gatech.edu}

\author[A. V. Maltsev]{Anna V. Maltsev}
\address{School of Mathematical Sciences,
Queen Mary University of London,
London E1 4NS, UK}
\email{annavmaltsev@gmail.com}


\begin{abstract}

We discuss explicit
landscape functions for quantum graphs.
By a ``landscape function'' $\Upsilon(x)$ we mean a function that
controls the localization properties of
normalized eigenfunctions $\psi(x)$ through a pointwise inequality of the form
$$
|\psi(x)| \le \Upsilon(x).
$$
The ideal $\Upsilon$ is a function that
\begin{enumerate}[a)]
\item
responds to the potential energy
$V(x)$ and to the structure of the graph in some formulaic way;
\item
is small in examples where
eigenfunctions are suppressed by the tunneling effect; and
\item
relatively large in regions
where eigenfunctions
may - or may not - be concentrated, as observed in specific examples.

\end{enumerate}

It turns out that the connectedness of a graph can present a barrier to the
existence of universal landscape functions in the high-energy r\'egime, as we show
with simple examples.  We therefore apply
different methods in different r\'egimes determined by the values of
the potential energy $V(x)$ and the eigenvalue parameter $E$.


\end{abstract}

\maketitle

\section{Introduction}

The overarching question that we investigate in this paper is how the graph structure impacts the behavior of eigenfunctions. A quantum graph is locally one-dimensional
and within the realm of Sturm-Liouville theory, but multidimensional features arise from the connectedness.  It
can be thought of as an intermediate case between one-dimensional and multidimensional models.

This article is an exploration of the degree to which explicit
landscape functions can be constructed for the eigenfunctions of quantum graphs.
The phrase ``landscape function'' was introduced by Filoche and Mayboroda in \cite{FiMa1}
to describe a method for locating where eigenfunctions of Schr\"odinger operators on domains and similar
partial differential equations tend to localize.  In particular, they used an
adapted torsion function to create their upper bound, which was a direct inspiration for our
\S \ref{landscape}, below.
Because other techniques can do a better job of localizing eigenfunctions in some circumstances, we
have chosen here to broaden the term
``landscape function'' $\Upsilon(x)$ to mean any function that
can be readily computed or estimated which
controls the localization properties of
normalized eigenfunctions $\psi(x)$ through a pointwise inequality of the form
$$
|\psi(x)| \le \Upsilon(x).
$$
The ideal $\Upsilon$ will be an explicit function simply expressed in terms of the eigenvalue $E$, the metric graph $\Gamma$,
and the potential energy $V(x)$ in
the Schr\"odinger equation living on it.  $\Upsilon(x)$ should vary over the graph and usefully distinguish
the regions where an eigenfunction may be large from those
where it must have small amplitude due to the tunneling effect.

The literature abounds with
techniques to obtain uniform $L^p$ estimates
of eigenfunctions of quantum Hamiltonians for $2 < p \le \infty$, notably Nelson's notion of hypercontractivity, as further developed by many later researchers, cf.\ \cite[\S X.9]{RS2}, \cite[\S 2]{Dav}, \S 2.
See also \cite{Dav83,BaHaTa,Sog}
for other approaches to pointwise bounds on eigenfunctions.
In \cite{Dav13} Davies showed that hypercontractive estimates can be adapted to the case of quantum graphs, as we shall recall
in Proposition~\ref{linftybd1}, below.

For differential
operators, several techniques have successfully been used to
construct landscape functions
that vary in useful ways over Euclidean domains or manifolds,
and related aproaches
will be explored here for quantum graphs.
The circumstances that determine which method is the most effective depend
heavily on the relationship between $V(x)$ and $E$, and to a lesser extent on the graph structure.  The strongest control is obtained in the
\emph{tunneling r\'egime}, where $V(x) > E$,
which is the subject of
\S \ref{Agmon}, using an
Agmon metric,
An explicit upper bound with tunneling decrease into a barrier is
stated in
Theorem \ref{tunneldecr}.  This section follows our
previous work \cite{HaMa}, but improves it by
extending its validity and by making the
constants explicit.
It is even possible to adapt the Agmon method to obtain landscape bounds when
$E > V(x)$, modestly, as we show later in Eq.~\eqref{AgmonUps}.
(For prior work controlling eigenfunctions of differential equations
with Agmon's method, we mention, for example,
\cite{Agm,HiSi,DFP}.)

The second established method uses the
maximum principle to prove inequalities in terms of functions satisfying other differential equations,
especially variants of the torsion function,
as in \cite{FiMa1,FiMa2,Ste}.
We innovate in \S \ref{landscape}
by replacing the torsion function by
something more explicit, consisting of functions of the form constant $+$ Gaussian on
a covering set of intervals and star graphs.  The covering can even in principle be made global
for the graph, although the upper bound will bcome trivial ({\em i.e.},
worse than the uniform bound)
on regions where $E \gg V$.
Examples show that the method based on maximum principles can work well
where $E > V(x)$ but only modestly.

For completeness, in later sections we work out bounds in the situations not covered
in \S\S \ref{Agmon}--\ref{landscape},
that is, when  $E \gg V$ and when $E \approx V(x)$.
Classical ODE methods are available to produce good
pointwise control of solutions, as we review, and for the transitional r\'egime where
$E \approx V(x)$ we are also able to use a variant of the Agmon method, to
give pointwise control of an eigenfunction by integrating over an enclosing ``window.''
In the high-energy r\'egime, there is a key difference
from the previous methods, however:  Whereas the Agmon method and the maximum principle allow one
to control an eigensolution on
an appropriate subset by its values on the boundary of the subset,
the only methods available in the high-energy r\'egime
are shooting methods.  That is, they take the value of a solution and
its derivative at a point and use them to control the solution as it moves along an edge.  Unfortunately,
as evidenced by Case Studies
\ref{scarycasestudy} and \ref{scarycasestudy2},
when a
path passes a vertex from one edge to the next,
the eigensolution on a succeeding vertex can set out with an uncontrolled
change in its derivative.
As a consequence, the bounds obtained from
classical ODE methods do not adapt as well to quantum graphs as do
those using the Agmon and maximum-principle methods.

A phase diagram delineating the different r\'egimes for constructing landscape functions
is depicted in Figure 1.

\begin{figure}\label{phase_diagram}
\begin{center}
\includegraphics[height=9cm]{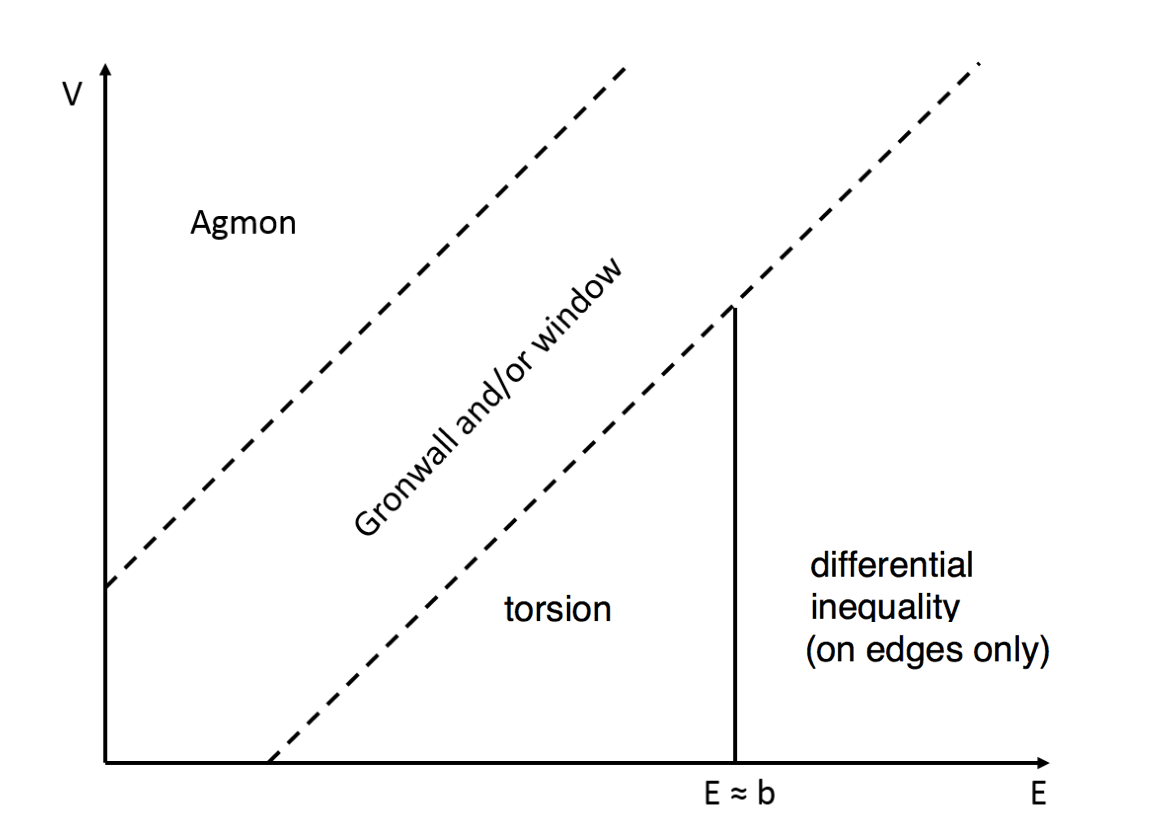}
\caption{Phase diagram of upper bounds on $|\psi|$.}
\end{center}
\end{figure}

\section{Assumptions on quantum graphs and some useful facts}\label{bkgdsection}

In this section we lay out some assumptions and review some facts about quantum graphs.
We recall that a quantum-graph eigenfunction $\psi(x; E)$ is an $L^2$-normalized
function that satisfies
\begin{equation}\label{EVE}
\left(- \frac{d^2}{dx^2} + V(x)\right) \psi(x;E) = E \psi(x;E)
\end{equation}
on the edges of a metric graph $\Gamma$,
and certain conditions at the vertices.
For simplicity, in this article we confine ourselves to
Kirchhoff (a.k.a.\ Neumann-Kirchhoff \cite{Ber}) vertex conditions,
according to which the sum of the outgoing derivatives at each vertex is $0$. We refer to \cite{BeKu,Ber,Post} for background and precise definitions of these operators.

We may assume without loss of generality that the graph has no leaves.
The Kirchhoff vertex condition at the end of a leaf reduces to the
standard Neumann boundary condition.  Any quantum-graph eigenvalue problem on a graph
$\Gamma$ with leaves can be restated on a larger graph $\hat{\Gamma}$ with no leaves,
where $\hat{\Gamma}$ consists of two copies of $\Gamma$ after identification of the corresponding end vertices of the
leaves.  The eigenfunctions on $\Gamma$ simply correspond to eigenfunctions on $\hat{\Gamma}$ which happen to be even
under the symmetry of swapping the two copies $\Gamma$ that compose $\hat{\Gamma}$.

As shown in Figure \ref{phase_diagram}, we will distinguish different parts of the graph based on the corresponding relationship between $V$ and $E$. This is captured in the following set of definitions:

\begin{defi}\label{TunRegDef}
For any finite $E \ge 0$ we refer to
$$
\mathcal{T}_E := V^{-1}((E, \infty))
$$
as the {\em tunneling region} (with respect to $E$), and to its complement
$$
\mathcal{C}_E := V^{-1}([0, E])
$$
as the {\em classically allowed region}.
Since eigenfunctions are expected to be more highly oscillatory where $E \gg V(x)$, we will sometimes
single out
{\em regions of low potential energy},
$\mathcal{C}_{E'} := V^{-1}([0, E'])$, where $E' < E$
(In physical parlance, these are the ``bottoms of the wells.'')
\end{defi}

Throughout the manuscript we make the following assumptions on the graph $G$ and potential $V$:
\begin{enumerate}
\item  The degrees of the vertices are uniformly bounded above by some
$d_{\max} < \infty$.
\item
Every edge
is at least as long as some fixed $L_{\min} > 0$.
\item  $V \ge 0$ and locally integrable
\item
For the eigenvalues $E$ we consider, the classically allowed region is compact.
\end{enumerate}

We shall have occasion below to invoke the maximum principle,
which is easy to extend
to the setting of quantum graphs (cf e.g. \cite{PoPr,Bak}).
We provide a version here that applies to
quantum graphs in a form that is convenient for our purposes.

\begin{lemma}\label{maxprin}
Let $H$ be a quantum-graph Hamiltonian
with $V(x) \ge 0$ on an open subset $\mathcal{S}$ of $G$.  Suppose that $w \in C^2$ and that $H w := - w'' + V(x) w$ on edges, with
``super-Kirchhoff'' conditions at the vertices $v$, that
\begin{equation}\label{superK}
\sum_{\bf e \sim v}w_{\bf e}^\prime(x_0^+) \ge 0,
\end{equation}
i.e., the sum of the outgoing derivatives of $w$ at a vertex is nonnegative.
If $H w \le 0$ on the edges contained in $\mathcal{S}$,
then $w_+ := \max(w,0)$ does not have a strict local maximum on $\mathcal{S}$.
\end{lemma}

\begin{proof}
(See also {\cite{Har}})
We follow a standard proof of the maximum principle for elliptic
partial differential equations,
taking special care at the vertices.

For this purpose we may assume that $w > 0$ at the putative maximum,
as the value 0 cannot logically be a strict local maximum value of $w_+$.
We next argue that it suffices to prove the maximum principle under the assumption that
$H w \le - \epsilon^2$  on $\mathcal{S}$ for some $\epsilon \ne 0$, since if $w$ has a strict local maximum
on $\mathcal{S}$,
then so does $w_\delta(x) := \exp(\delta x) w(x)$ for sufficiently small
$|\delta(x)|$, at a point $x_1 \in \mathcal{S}$.  But $H w_\delta(x) = \exp(\delta x)\left(-\delta^2 w - \delta w' + Hw\right)$,
and therefore for $\delta$ of sufficiently small magnitude and with the same sign as $w'(x_1)$
(supposing that $w'(x_1) \ne 0$),
this will be strictly negative in a neighborhood of $x_1$.

Thus we posit without loss of generality
that $H w \le - \epsilon^2$ for some $\epsilon > 0$.
If we suppose that $w$ is maximized at some
$x_0$ interior to an edge,
then $w^\prime(x_0) = 0$ and $w^{\prime\prime}(x_0) \le 0$, but this contradicts the
assumption that $H w \le - \epsilon^2$.  If on the
other hand the maximizing $x_0$ is a vertex, then for each edge ${\bf e}$
emanating from $x_0$, $w_{\bf e}^\prime(x_0^+) \le 0$.  Because of the super-Kirchhoff conditions, if for any edge,
$w_{\bf e}^\prime(x_0^+) < 0$, there must be at least one other edge $e'$ on which $w_{e'}^\prime(x_0^+) > 0$, which would contradict maximality.  Therefore $w_{\bf e}^\prime(x_0^+) = 0$ for all edges ${\bf e}$, and a necessary condition for maximality is again that
$w_{\bf e}^{\prime\prime}(x_0^+) \le 0$.  This, as before, would contradict $H w \le - \epsilon^2$.
\end{proof}

As mentioned above, eigenfunctions of quantum graphs are bounded above in the $L^\infty$ sense, using hypercontractive
(heat-kernel)
estimates.
A bound of this type was provided in
Lemma 4.1 of \cite{Dav13}:
\begin{prop}\label{linftybd0}
(Davies)
Assume that $\Gamma$ consists of a finite number of finite edges, and
let $\psi_\ell(x) := \psi(x;E_\ell)$ for a particular eigenvalue $E_\ell$.
Then the $L^2$ normalized
eigenfunctions $\psi_j(x)$ satisfy
\begin{equation}\label{BVE}
\sum_{E_j \le E}{\|\psi_j\|_{L^\infty({\bf e})}^2} \le C^2 \sqrt{E - \inf_{\Gamma}(V)},
\end{equation}
where the constant $C>0$ depends only on the graph $\Gamma$.
\end{prop}
For an individual eigenfunction $\psi_j$ it follows that if $E_j \le E$, then
\begin{equation}\label{BVE2}
|\psi_j(x)| \le C (E - \inf_{\Gamma}(V))^{1/4}.
\end{equation}

We next provide a more specific variant of Proposition
\ref{linftybd0}.  The key point in
its proof was
a comparison between the
free heat kernel (i.e., replacing $V$ with $0$) and a finite multiple
of the free heat kernel after additional Neumann boundary conditions
have been imposed at the ends of all the edges, thereby effectively disconnecting the graph.
To include a potential energy $V(x) \ge V_{\min}$, the kernel
of $\exp(-t H)$ can then be
bounded above by $\exp(-V_{\min}t)p_{\bf e}(t,x,y)$
according to a standard
argument using the Lie-Trotter product formula.
(See for example Lemma 1.1 of \cite{Dav73}).

The heat kernel on $\Gamma$ is of the form
\[
p_{\Gamma}(t,x,y) = \frac{1}{|\Gamma|}
+ \sum_{n=1}^\infty{\exp\left(-\lambda_n t\right) \phi_n(x)\phi_n(y)}
\]
where the eigenfunctions $\phi_n$ are normalized in $L^2(\Gamma)$, and in
particular the decoupled heat kernel with Neumann conditions on an
edge ${\bf e}$ is
\[
p_{\bf e}(t,x,y) = \frac{1}{|{\bf e}|} \left(1 + 2 \sum_{n=1}^\infty{\exp\left(-\left(\frac{n \pi}{|{\bf e}|}\right)^2 t\right) \cos\left(\frac{n \pi}{|{\bf e}|} x\right)\cos\left(\frac{n \pi}{|{\bf e}|} y\right)}\right).
\]
choosing the coordinate $x$ on ${\bf e}$ in a convenient way.  The heat-kernel bounds we shall require will estimate the heat-kernel above on each edge by $L^\infty$ norms of the eigenfunctions,
so let us define
\begin{equation}\label{upperhk}
\hat{p}_{\Gamma}(t, {\bf e}) := \frac{1}{|\Gamma|}
+ \sum_{n=1}^\infty{\exp\left(-\lambda_n t\right) \|\phi_n\|_{L^\infty({\bf e})}^2},
\end{equation}
where ${\bf e}$ ranges over the edges of the graph $\Gamma$, and we similarly write
$\hat{p}_{{\bf e}}(t, {\bf e})$ when considering a decoupled edge ${\bf e}$.  On a decoupled edge,
\begin{align}\label{thetabd}
p_{\bf e}(t,x,y) &\le \hat{p}_{{\bf e}}(t, {\bf e})\nonumber\\
&\quad\quad= \frac{1}{|{\bf e}|} \left(1 + 2 \sum_{n=1}^\infty{\exp\left(-\left(\frac{n \pi}{|{\bf e}|}\right)^2 t\right)}\right)\nonumber\\   
&\quad\quad=  \frac{1}{|{\bf e}|} \vartheta_3\left(0,\exp\left(-\left(\frac{\pi}{|{\bf e}|}\right)^2 t\right)\right).
\end{align}

By expanding $\exp(-t H)$ in eigenfunctions
and using the Lie-Trotter formula and
results in \cite{Dav13}, we find that
\begin{equation}
\sum_{E_j \le E}{|\psi_j(x)|^2} \le C^2 e^{(E - \inf_{\Gamma}(V)) t} \hat{p}_{{\bf e}}(t, {\bf e})
\end{equation}
for some $\Gamma$-dependent positive constant $C$.  In fact,
the theta function in \eqref{thetabd}
is dominated
by a larger but more elementary quantity, giving
\[
\sum_{E_j \le E}{|\psi_j(x)|^2} \le C^2 \, \frac{e^{(E - \inf_{\Gamma}(V)) t}}{|{\bf e}|}\left(1 + \frac{|{\bf e}|}{\sqrt{\pi t}}\right),
\]
and choosing $t = \frac{1}{2(E - \inf_{\Gamma}(V)}$ (which is the minimizing value
if one ignores ``$1 +$'' on the right side), one gets a version of \eqref{BVE}, {\em viz}.,
\begin{equation}\label{OLDBVE}
\sum_{E_j \le E}{\|\psi_j\|_{L^\infty({\bf e})}^2} \le C^2 \left(\sqrt{\frac{2 e (E - \inf_{\Gamma}(V))}{\pi}} +  \frac{\sqrt{e}}{|{\bf e}|}\right).
\end{equation}
In the following theorem we work out more explicit constants.

\begin{thm}\label{linftybd1}
Assume that $\Gamma$ consists of a finite number of finite edges, and set
\[
M := 2 \max_{[\pi,\infty]}\left(\frac{1}{1 - \frac{\sin x}{x}}\right) \doteq 2.29456\dots.
\]
Then
\begin{equation}\label{hkub}
\hat{p}_{\Gamma}(t, {\bf e}) \le \frac{1}{|\Gamma|} + \frac{1}{|{\bf e}|}\left(3 \sum_{E_j < \frac{\pi}{|{\bf e}|}} {e^{-E_j t}}
+  M \sum_{E_j \ge \frac{\pi}{|{\bf e}|}} {e^{-E_j t}}\right)
\end{equation}
and
\begin{equation}\label{hkub2}
\sum_{\bf e}\hat{p}_{\Gamma}(t, {\bf e}) \le \frac 3 2 \sum_{\bf e}\hat{p}_{\bf e}(t, {\bf e}).
\end{equation}
In particular, the estimates of Prop. \ref{linftybd0} hold with
$C^2 \le \frac{3 m}{2 \min|{\bf e}|}$, where $m$ is the number of edges of $\Gamma$.
\end{thm}

\begin{rem}
The constant given in the last line of the theorem is quite crude, and can be greatly improved in
all of the examples we have examined.
Indeed, the free heat kernel for $\Gamma$ can often be estimated directly, and,
{\em e.g.}, for the regular tetrahedral graph in Case Study \ref{tetrah},
\begin{equation}\label{hkubcstis1}
\hat{p}_{\Gamma}(t, {\bf e}) < \hat{p}_{{\bf e}}(t, {\bf e}).
\end{equation}
The fact that the constant in this case $\le 1$ can be understood with reference to
the symmetry group, which allows an averaging over the edges, and the inequality is
strict because the constant term in the heat kernel is greater for $p_{\bf e}$ than for $p_{\Gamma}$.
We conjecture that:
{\em \eqref{hkubcstis1} holds true for every finite metric graph more complicated than a single
edge, and that even with infinitely
long edges a similar domination holds.}
\end{rem}

\begin{proof}
Because the eigenfunctions of the free quantum graph (setting $V=0$) are of the form
$\psi(x) = N_{\bf e} \cos(\sqrt{E} (x - \phi_{\bf e}))$ on each edge,
the $L^2$ and $L^\infty$ norms of $\psi$ on a given edge are related by optimizing
the integral
\[
\|\psi\|_{L^2{\bf e}}^2 = \int_{\bf e}(\cos(\sqrt{E}(x - \phi))^2 dx
\]
with respect to $\phi$, which by a calculation yields
\begin{equation}\label{cosbd}
\|\psi\|_{L^2({\bf e})}^2 \ge  \frac{|\bf e|}{2} \left(1 - \frac{\sin(\sqrt{E} |\bf e|)}{\sqrt{E} |\bf e|)})\right) \|\psi\|_{L^\infty({\bf e})}^2,
\end{equation}
provided that $|\psi(x)|$ attains a maximum in ${\bf e}$, which is guaranteed by the Sturm separation theorem when
$\sqrt{E} |{\bf e}| \ge \pi$.  This is the origin of the constant $M$ in the theorem.
For edges not containing a local maximum of $|\psi(x)|$,
a straightforward bound is obtained by appealing to the concavity of $|\psi(x)|$ on intervals
where $\psi$ does not change sign, for which comparison with linear interpolation
gives
\[
\|\psi\|_{L^\infty({\bf e})}^2 \le \frac{3}{|{\bf e}|}\|\psi\|_{L^2({\bf e})}^2.
\]
The possibility of a sign change would only improve this upper bound.
The $L^\infty$ norms of the eigenfunctions on the edges are thus uniformly bounded from above
by the $L^2$ norms, so that
with
\begin{equation}
\|\sum_\ell |\psi_{k_\ell}|^2\| \le \sum_\ell \|\psi_{k_\ell}\|^2,
\end{equation}
one readily obtains \eqref{hkub}.
\smallskip
To derive \eqref{hkub2}, recall that the imposition of Neumann boundary conditions lowers the eigenvalues, and consequently
increases the factors $e^{- \lambda_k t}$ for $t > 0$.  Meanwhile, in order to account properly for the full eigenspaces associated with
$\Gamma$, it is necessary to include all of the edges.
(For the sake of a simple formula we have replaced $M$ by $3$ and have not optimized the constant term.)
\end{proof}

Another inequality that we can adapt to quantum graphs with a
simple
proof is the Harnack inequality.

\begin{thm}[Harnack inequality for quantum graphs]
Let $U$ be an open subset of $\Gamma$ and let $W \subset U$ be connected and compact.
Then there exists a constant $C$ depending only on $U$, $W$, $V(x)$, and $E$, such that
every real-valued $\psi(x)$ defined on $U$, which never vanishes and satisfies
\[
\text{\rm sgn}(\psi(x))(- \psi''(x) + (V(x) -E) \psi) \ge 0
\]
on the edges and Kirchhoff conditions at the vertices,
obeys the inequality
\[
\frac{\max_W |\psi|}{\min_W |\psi|}\leq C.
\]
\end{thm}
\begin{proof}
  We may ssume $\psi >0$. Abbreviating $H f := - f'' + V f$ as usual,
  \begin{equation}\begin{split}
  (H-E)\ln \psi &= - \frac{d}{dx}\, \left( \frac{\psi'}{\psi}\right) + (V - E)\ln \psi \\
  &
  = \frac{1}{\psi}\,(H- E)\psi + \left( \frac{\psi'}{\psi}\right)^2 + (V - E)(\ln \psi - 1).
  \end{split}\end{equation}
  By assumption the first term on the right is nonnegative, and so for all $x$ (other than vertices) in $U$, we get
  \begin{equation}
  \left( \frac{\psi'}{\psi}\right)^2 \leq -\frac{d^2}{dx^2} \ln \psi + (V - E).
  \end{equation}
  Let $r = \ln (\frac{\psi(x_2)}{\psi(x_1)})$ for some fixed pair of points $x_{1, 2}\in W$
  (for example, $x_2$ maximizing $\psi$ and $x_1$ minimizing $\psi$).
  Then if $P$ is any path from $x_1$ to $x_2$,
  \[
  r^2 = \left( \int_{P}\frac{\psi'(t)}{\psi(t)}\, dt\right)^2 \leq |P| \int_P\left( \frac{\psi'(t)}{\psi(t)}\right)^2 dt.
  \]
  Let $\tilde P = P \cup J$ where $J = \cup I_i$, and $I_i$ are short intervals of two kinds
  adjacent to $P$
  (i.e. they are short enough that they do not reach the next vertex):
  \begin{enumerate}
    \item Short extensions beyond $x_{1,2}$
    \item Some neighborhoods of the vertices, i.e. including little bits of edges whose vertices lie in $P$
  \end{enumerate}
  Now let $\eta$ be a piecewise
  $C^1$ function such that $\eta := 1$ on $P$ and $\eta:= 0$ on $\tilde P^c$.
  (Specifically, $\tilde P$ could be chosen as $\{x \in U: \dist(x, P) < L_{\text{min}}/2 \}$, and
  $\eta$ as a linear ramp going from $1$ to $0$ as $x$ goes from $P$ to $\partial \tilde P$.)
  Then
  \[
  r^2 \leq |P|\int_{\tilde P}\eta^2 \left( \frac{\psi'}{\psi}\right)^2 \leq |P| \int_{\tilde P}\eta^2 \left(-\frac{d^2}{dx^2}\ln \psi + V - E\right)
  \]
  We now integrate by parts and use the fact that the contributions at the vertices add up to zero by Kirchhoff, leaving
  \begin{equation}\begin{split}
  \int_{\tilde P}\eta^2\left( \frac{\psi'}{\psi}\right)^2
  &
  \leq \int_{\tilde P} \eta^2(V - E) + \int_{\tilde P}2\eta'\eta\, \frac{\psi'}{\psi}\\
  &
 \leq  \int_{\tilde P} \eta^2(V - E) + \frac{1}{\alpha}\int_{\tilde P} (\eta ')^2 + \alpha \int_{\tilde P}\left(\eta \frac{\psi'}{\psi}\right)^2.
  \end{split}\end{equation}
  Choosing $\alpha = 1/2$ we obtain
 \[
 \int_{\tilde P}\eta^2\left( \frac{\psi'}{\psi}\right)^2
 \leq 2 \int_{\tilde P} \eta^2(V - E) + 4\int_{\tilde P}(\eta ')^2,
 \]
 which is independent of $\psi$ as claimed.
\end{proof}

A final tool we adapt to quantum graphs is a lower-bound inequality of Boggio (more often attributed to Barta; see
\cite{FHT99}
 for some discussion of the contribution of Boggio \cite{Bog}), {\em viz}.,
i.e.\ if $\Delta$ is the Dirichlet Laplacian on a domain and $v(x)>0$ is a suitably regular function,
 then, in the weak sense,
$$
-\Delta \ge \frac{- \Delta v(x)}{v(x)}.
$$

Since the graph Laplacian is more analogous to a domain's Neumann Laplacian than to its Dirichlet Laplacian, it may be surprising that
Boggio's inequality
extends without complications:
\begin{lemma}\label{Boggio}
Let $\Gamma_0$ be a quantum graph with Kirchhoff or Dirichlet boundary conditions at vertices, possibly independently assigned.  Suppose that $\Phi > 0$ is a $C^2$ function
on the edges and satisfies super-Kirchhoff conditions
\eqref{superK} at all vertices.  Then for every $f \in H^1(\Gamma)$,
\[
\sum_{{\bf e}\in \Gamma_0}\int_{\bf e}{|f'(x)|^2} \ge \sum_{{\bf e}\in \Gamma_0}\int_{\bf e}{|f(x)|^2 \left(\frac{-\Phi''(x)}{\Phi(x)}\right)}.
\]
\end{lemma}

\begin{proof}
For notational simplicity, the proof is carried out in the case where $f$ is real valued.  According to Picone's inequality,
\begin{align*}
(f'(x))^2  &\ge \Phi'(x) \frac{d}{dx} \left(\frac{(f(x))^2}{\Phi(x)}\right)\\
& =  \frac{d}{dx}\left(\Phi(x) \frac{d}{dx}\left(\frac{(f(x))^2}{\Phi(x)}\right) \right) + (f(x))^2\left(\frac{-\Phi''(x)}{\Phi(x)}\right).
\end{align*}
When the first term in the last line is integrated on an edge ${\bf e}$, it contributes
$$-2 f(0+) f'(0+) + (f(0+))^2 \frac{\Phi'(0+)}{\Phi(0+)}$$
in the outgoing sense at both of the vertices bounding ${\bf e}$.  When all such contributions are summed at a given vertex,
the result is nonnegative according to the assumptions on $f$ and $\Phi$.
\end{proof}

\section{Landscape upper bounds on tunneling regions,\\
 using Agmon's method}\label{Agmon}

It is in the tunneling r\'egime $\mathcal{T}_E$ that
the estimation of eigenfunctions in terms of a landscape function
is at the same time the most explicit
and the tightest when compared with examples.  We thus start
by recalling and sharpening
some bounds derived with Agmon's method, which were first established for quantum graphs
in \cite{HaMa}.

The two central lemmas in \cite{HaMa} can be distilled into the
following pointwise identities for an Agmon function $F$, a smooth cutoff $\eta$, and a
real-valued function
$\psi$ satisfying $(H-E) \psi = 0$ on ${\rm supp}(\eta)$.  First:
\begin{equation}\label{1stid.a}
F^2(x)\eta(x) \psi(x) \left(H - E \right) \eta(x) \psi(x) =
F^2(x)\left(- \eta''(x) \psi^2(x)- \eta \psi(x) \eta'(x)  \psi'(x) \right),
\end{equation}
where the quantity on the right is supported
within $\rm{supp}(\eta') =: \mathcal{S}$, and can therefore be estimated in terms of
$\|\eta''\|_\infty$, $\|\eta'\|_\infty$, $\sup(F)\chi_{ \mathcal{S}}$,
and $\|\psi\|_{H^1(S)}$.  With a little algebraic juggling, we can rewrite
\eqref{1stid.a} so that
the derivatives $\psi'$ and
$\eta''$ do not appear:
\begin{equation}\label{1stid.b}
F^2(x)\eta(x) \psi(x) \left(H - E \right) \eta(x) \psi(x) = \psi^2(x) \left(\eta^\prime(x)\left(\eta(x) F^2(x)\right)^\prime\right) - G^\prime(x),
\end{equation}
except inside $G(x) = \eta(x) \eta^\prime(x) F^2(x) \psi^2(x)$, which
will be arranged to integrate to 0 on any edge
by requiring
the support of $\eta^\prime$ to lie within the edge.
In \eqref{1stid.b}
the quantity on the right
can therefore be estimated in terms of
$\|\eta'\|_\infty$, $\sup(F)\chi_{ \mathcal{S}}$,
$\sup(F')\chi_{ \mathcal{S}}$,
and $\|\psi\|_{L^2(S)}$.  In particular, this will allow us to relax the assumption that
$\eta \in C^2$ and choose it to be a ramp function below.

By a second direct calculation,
\begin{align}\label{2ndid}
F^2(x)\eta(x) &\psi(x) \left(H - E \right) \eta(x) \psi(x)
=\nonumber\\
&\quad \left( \left(F \eta \psi \right) ^\prime \right)^2 + {\left(V - E - \left(\frac{F^\prime}{F}\right)^2 \right)} \left(F \eta \psi \right)^2 -H^\prime(x),
\end{align}
where $H=F^2 \eta \psi \left( \eta \psi \right)^\prime$ will produce boundary contributions
when integrated,
but if $F$ is continuous and $\psi$ satisfies Kirchhoff conditions, they will sum to $0$.

Suppose initially that
for some $\ell > 0$ the finite, closed interval $[a-\ell,b+\ell ]$ is contained within an edge belonging to
$\mathcal{T}_E$.
In order to obtain estimates on $I = [a,b]$, let
\begin{alignat}{2}
F_E(x) &= 1, &\quad x \notin (a,b)\nonumber\\
F_E(x) &= e^{\min{\int_a^x \sqrt{V - E},\int_x^b \sqrt{V - E}}},
&\quad  x \in [a,b]
\end{alignat}
We let $\eta$ be a linear
ramp on $[a-\ell, a]$ and $[b, b+\ell]$, with $\eta = 0$ on $\mathbb{R}
\backslash [a-\ell, b+ \ell]$ and $\eta = 1$ on $I$. Then
equations \eqref{1stid.b} and \eqref{2ndid} yield that
\begin{equation}\label{e:intGamma}
\int_\Gamma \psi^2(x)\left(\eta'(x)\left(\eta(x)F_E^2(x)\right)'\right)
\ge
\int_\Gamma \left(\left(F_E\eta\psi\right)'\right)^2.
\end{equation}
Now fixing $x\in I$ we apply Cauchy-Schwarz to the right side:
\begin{equation}\label{e:ptwsbd}\begin{split}
\int_\Gamma \left(\left(F_E\eta\psi\right)'\right)^2
 &=
 \int_{a-\ell}^x \left(\left(F_E\eta\psi\right)'\right)^2 + \int_x^{b+\ell} \left(\left(F_E\eta\psi\right)'\right)^2
 \\
 &
 \geq \frac{ (\int_{a - \ell}^x \left(F_E\eta\psi\right)')^2}{x -a + \ell} + \frac{ (\int_x^{b+ \ell}
 \left(F_E\eta\psi\right)')^2}{b + \ell - x}
 \\
 &
= \frac{(F_E(x)\psi(x))^2}{x -a + \ell}+ \frac{(F_E(x)\psi(x))^2}{b + \ell - x}
\\
&
= \psi(x)^2 \left(\frac{F_E(x)^2(b - a + 2\ell)}{(x - a + \ell)(b + \ell - x)}\right).
\end{split}\end{equation}
This yields the estimate
\begin{equation}\label{e:calc}\begin{split}
|\psi(x)| &\leq \sqrt{\frac{(x - a + \ell)(b + \ell - x)}{b - a + 2\ell}}\\
&\times \frac{1}{F_E(x)} \left(\int_{\supp \eta'} \psi^2(y)\left(\eta'(y)\left(\eta(y)F_E^2(y)\right)' \right)dy\right)^{\frac 1 2 } \\
& = \sqrt{\frac{(x - a + \ell)(b + \ell - x)}{b - a + 2\ell}}\, \frac{1}{F_E(x)} \left(\int_{\supp \eta'} \psi^2(y) (\eta'(y))^2 dy\right)^{\frac 1 2 } \\
&\leq \sqrt{\frac{(x - a + \ell)(b + \ell - x)}{b - a + 2\ell}}\, \frac{\|\psi\|_{L^2([a-\ell, a]\cup[b, b+\ell])}}{\ell \, F_E(x)}.
\end{split} \end{equation}

We now extend this argument in two ways.  The first is to potentially allow the interval to be infinite, as was the case in \cite{HaMa}. We may parametrize
$I$ as  as $[a,\infty)$, in which case
$F_E$ can be simply defined on $I$ as
\[
F_E(x) := e^{\int_a^x\sqrt{V-E}}.
\]

In this case we can drop one of the contributions to the first line of \eqref{e:ptwsbd}, obtaining
\[
\int_\Gamma ((F_E\eta\psi)')^2 \geq \frac{(F_E(x)\psi(x))^2}{x - a + \ell}
\]
and thus for $x > a$,
\[
|\psi(x)|\leq \sqrt{x-a + \ell}\, F^{-1}_E(x) \frac{\|\psi\|_{L^2([a-\ell, a])}}{\ell}.
\]

Secondly, we extend the analysis to connected regions of the graph on which $V-E > 0$ as follows. Since we assume $V$ to be continuous, we know that the set $\mathcal{T}_E$ is open. It may consist of disconnected components, in which case we may restrict ourselves to working on one component at a time, so
without loss of generality we may assume that $\mathcal{T}_E$ is connected. Let the boundary of $\mathcal{T}_E$ (henceforth denoted $ \partial_{\mathcal{T}_E}$) be $\{b_1, \dots , b_m\}$. Note that $\partial \mathcal{T}_E$ is a finite collection of points, because we assume that $\gamma \backslash \mathcal{T}_E$ is compact, all degrees are finite, and all edges have a minimum length. We define
\begin{align}
F_E(x) &= \exp\left(\min_{1 \leq j \leq m}\; \min_{P: \text { paths }
b_j \text{ to } x } \int_P \sqrt{V - E}\right)\quad &\text{for }
x \in \mathcal{T}_E
\\
F_E(x) &=1 \quad &\text{for } x \notin \mathcal{T}_E
\end{align}
By construction, $F_E$ is again continuous.
For $x \in \mathcal{T}_E$ we can think of $F_E(x)$ as defining an
\emph{Agmon metric}
on $\mathcal{T}_E$,
\begin{equation}\label{Agdist}
\rho_A(x,y;E) := \min_{P: \text { paths } y \text{ to } x } \int_P \sqrt{V - E}.
\end{equation}
If $S$ is a set, $\rho_A(x,S,E)$ will denote the infimum of $\rho_A(x,y;E)$ for $y \in S$.

To define $\eta$, for each $j$ we parametrize the part of the edge containing $b_j$ and lying outside of $\mathcal{T}_E$ with $b_j$ mapped to 0. Then $\eta$ is taken as a ramp on each of $m$ segments $[0, \ell]$ associated to each point in $\partial_{\mathcal{T}_E}$ (denote them $[0, \ell]_{b_j}$) so that $\eta = 1$ on $\mathcal{T}_E$ and $\eta = 0$ on $\mathcal{T}_E^c \backslash \cup_j [0, \ell]_{b_j}$. This construction yields $\eta'(x) = -1/\ell$ on each of $[0, \ell]_{b_j}$.
With $\eta$ and $F_E$ in place, we carry out a similar calculation. Let $P$ be any path from any of the points $b_j \in \partial_{\mathcal{T}_E}$ to the point $x \in \mathcal{T}_E$.
\begin{equation}
\int_\Gamma ((F_E \eta \psi)')^2 \geq \int_P ((F_E \eta \psi)')^2 \geq \frac{\left(\int_P (F_E\eta \psi)'\right)^2}{|P|} = \frac{\left(F_E(x) \psi(x)\right)^2}{|P|}.
\end{equation}
We can then minimize over paths $P$ to obtain
an upper bound,
which decreases exponentially into the tunneling region.  This proves:
\begin{thm}\label{tunneldecr}
For $x \in \mathcal{T}_E$ with
$\dist(x, \partial \mathcal{T}_E) \ge \ell$,
\begin{equation}\label{Agub}
\left|\psi(x)\right| \leq \frac{\sqrt{\dist(x, \partial \mathcal{T}_E)}\|\psi\|_{L^2(\cup_{j= 1}^m \; [b_j, b_j + \ell])}}{\ell}
\exp(-\rho_A(x,\partial \mathcal{T}_E;E)).
\end{equation}
\end{thm}
For a normalized wavefunction we can simplify by bounding
$\|\psi\|_{L^2(\cup_{j= 1}^m  [b_j, b_j + \ell])}$ above by 1.
We caution that, unlike the upper bound of Theorem~\ref{tunneldecr}, the magnitude of the wave function itself may, and frequently does, change monotonically
at an exponential rate
throughout a barrier.  Of course, if it does so, normalization forces it to be
exponentially small on one side or other of the barrier.

In some circumstances, a different choice of $F$ can provide a slightly improved upper bound with Agmon's method.

Now fix some $\delta > 0$, and consider the set
$\mathcal{T}_{E+ \delta} \setminus \mathcal{T}_{E+ 2 \delta}$.
Each connected component of this set contains
a vertex-free interval of length $\ge L(\delta)$ for some $L(\delta)> 0$,
the value of which we consider among the ``accessible'' properties of a quantum graph.

Integrating \eqref{2ndid} and letting $Q^2 := \left(V - E - \left(\frac{F_{E-\delta}'}{F_{E-\delta}}\right)^2\right) \ge \delta$, we get
\begin{align}\label{lowerbd}
\int_{\Gamma}{\left(((\eta F_{E-\delta} \psi)^\prime)^2 + Q^2(\eta F_{E-\delta} \psi)^2\right)}
&= \int_{\Gamma}{Q\left(\frac{((\eta F_{E-\delta} \psi)^\prime)^2}{Q} + Q(\eta F_{E-\delta} \psi)^2\right)}\nonumber\\
&\ge \sqrt{\delta} \int_{I_0}{\left(\frac{((\eta F_{E-\delta} \psi)^\prime)^2}{Q} + Q(\eta F_{E-\delta} \psi)^2\right)}\nonumber\\
&\ge \sqrt{\delta} \left|\int_{I_0}{\left((\eta F_{E-\delta} \psi)^2\right)^\prime}\right|,
\end{align}
where ${I_0}$ is any subset of $\mathcal{T}_{E+ \delta}$.  (The final line used the
arithmetic-geometric mean inequality, $a^2 + b^2 \ge \pm 2 a b$.)
In order to estimate $\psi(x)$ for $x \in \mathcal{T}_{E+ 2 \delta}$, we
make a specific choice of $\eta \in C^2$ and $I_0$ as follows.
\begin{enumerate}
\item
$I_0$ is a vertex-free interval of length $\ge L(\delta)$.
\item
$\supp(\eta^\prime(x)) \subset
I_0 \cup I_1 \subset \mathcal{T}_{E+ \delta} \setminus \mathcal{T}_{E+ 2 \delta}$, where
$I_1$ is a finite (possibly empty) union of disjoint vertex-free intervals,
such that any path from
$x$ to the complement of $\mathcal{T}_{E+ \delta}$ passes through
$I_0 \cup I_1$.
\item
$\eta(x) = 1$
\item
$\eta(y) = 0$ for all $y$ that cannot be connected to $x$ without passing through $I_0$.
\end{enumerate}

Applying the Fundamental Theorem of Calculus to the lower side of
\eqref{lowerbd} and invoking the equivalence of  \eqref{1stid.b} and \eqref{2ndid}, we see that
\begin{equation*}
\left(F_{E-\delta}(x) \psi(x)\right)^2 \le \frac{1}{\sqrt{\delta}} \int_{I_0}{\psi^2 F_{E-\delta}^2 \left(\eta'^2 + (\eta^2)' \left(\frac{F_{E-\delta}'}{F_{E-\delta}}\right)\right)}.
\end{equation*}
The smoothness required of  $\eta$
can now be relaxed by passing to a sequence of $\eta_k$ tending uniformly to linear ramp functions
increasing from $0$ to $1$ on a subinterval of $I_0$
of length at least $L(\delta)$.  Hence we conclude that
\begin{align*}
\left(F_{E-\delta}(x) \psi(x)\right)^2 &\le  \frac{1}{\sqrt{\delta}}
\max_{I_0}\left(F_{E-\delta}^2 \left(\left(\frac{1}{L(\delta)}\right)^2 + \frac{1}{L(\delta)} \sqrt{V-E - \delta} \right)\right) \|\psi\|_{I_0}^2\\
&\le  \frac{1}{\sqrt{\delta}} e^{2 \sqrt{\delta} L(\delta)}\left(\left(\frac{1}{L(\delta)}\right)^2 + \frac{\delta}{L(\delta)}\right) \|\psi\|_{I_0}^2.
\end{align*}

\section{Construction of landscape functions on a graph\\ via
a simplified torsion function}\label{landscape}

Here and in \S \ref{trans} we shall discuss ways to construct
landscape functions valid when $E \ge V(x)$ (but not by too much),
thus complementary to the bounds of
\S \ref{Agmon}.

For scalar Schr\"odinger operators on domains, the
original choice by Filoche and Mayboroda for their
 ``landscape function,''
is a
sufficiently large multiple of a
positive solution $T(x)$ of
\begin{equation}\label{stdlsf}
(-\Delta + V(x))T(x) = 1,
\end{equation}
e.g., \cite{FiMa1,FiMa2,Ste}. This is a
Schr\"odinger variant of the
{\em torsion function} (cf. \cite{BLMS,RaSt,vdB12,vdB17})
A sufficiently large multiple of $T(x)$ will provide a pointwise bound on an eigensolution
$|\psi(x)|$
on some region $R$, through a maximum-principle argument.
The bound will depend on the eigenvalue and on
the values of $|\psi(x)|$ on $\partial R$.

There are two common drawbacks to landscape functions of torsion-function type.
The first is that,
typically, such landscape functions become trivial for large eigenvalues $E$, by which we mean that
the upper bound thus obtained may on some regions exceed
known uniform upper
bounds on $\|\psi\|_\infty$, e.g.,  as in Proposition~\ref{linftybd1}.
In this situation the upper bounds usually also lack
useful dependence on the position $x$.
This is an intrinsic difficulty
for the method in a region where the eigenfunction oscillates.  It is hard to see how
a necessarily positive upper bound will take full avantage of the
fact that such an eigenfunction has zeroes.
(An alternative and more effective
approach to pointwise control
of rapidly oscillating eigenfunctions
incorporates their derivatives, cf.\
Theorem \ref{vopbd}, below.)

Consider for instance
the simplest situation, an ordinary differential equation
with periodic boundary conditions, $\psi(x+L) = \psi(x)$.  At large energies $E \gg V(x)$ we can approximate
by dropping $V$,
so that the normalized real-valued eigenfunctions are well approximated by
\[
\sqrt{\frac{2}{L}} \cos(\sqrt{E} x - \phi), \quad E = \left(\frac{2 \pi m}{L}\right)^2,
\]
and by appropriate choice of the phase $\phi$ the position of the maximal value
can be placed at will.
In addition to this elementary limitation
on the use of landscape functions, when we adapt them to
quantum graphs
there are further barriers to their use arising from the connectedness of the graph,
as shown in Case Studies
\ref{scarycasestudy} and \ref{scarycasestudy2}.

A second drawback to landscape functions based on \eqref{stdlsf} is that, usually,
the torsion function and its variants are
only computationally known.  As we shall elaborate below in the context of quantum graphs,
however, due to the maximum principle it suffices in lieu of
\eqref{stdlsf} to have an inequality
\begin{equation}\label{lsfineq}
H T = (-\Delta + V(x))T(x) \ge 1,
\end{equation}
The flexibility of an inequality allows
more accessible or even explicit choices of landscape functions,
without losing qualitative features.

Before showing how to construct explicit, elementary functions satisfying
\eqref{lsfineq} on
quantum graphs, which is done below, let us describe how $T$ can be used to provide a landscape function in two different ways.

In both cases we suppose that \eqref{lsfineq} holds on some $\Gamma_0 \subset \Gamma$, with Kirchhoff conditons
at the vertices of $\Gamma_0$

We consider
\[
W(x,\pm) := \pm \psi(x) - E \|\psi\|_{L^\infty(\Gamma_0)} T(x),
\]
where we shall consider both signs in order to bound $|\psi|$.
We see that
\begin{equation}
H W(x,\pm) = E \left(\pm \psi(x) - \|\psi\|_{L^\infty(\Gamma_0)}\right) \le 0.
\end{equation}
We now apply the maximum principle Lemma \ref{maxprin} to $W(x,\pm)$
for both signs,
concluding that $W(x) := |\psi(x)| - E \|\psi\|_{L^\infty(\Gamma_0)} T(x)$ is maximized on the boundary
of the region on which \eqref{lsfineq} holds.
We thus obtain
\begin{equation}\label{ttlsfn}
|\psi(x)| \leq \max_{x \in \partial \Gamma_0}W_+(x) + E \|\psi\|_{L^\infty(\Gamma_0)}T(x),
\end{equation}
and hence if $W = 0$ on $\partial \Gamma_0$ then
$$
\Upsilon_{\max \rm princ.}(x) := E \|\psi\|_{L^\infty(\Gamma_0)}T(x)
$$
is a landscape function in the sense of
\cite{FiMa1,Ste}.  Of course, this is only interesting for $x$ such that $E T(x) < 1$ or when $\psi$ is known
{\em a priori}
to be small on $\Gamma_0$.

The second way to build a landscape function out of $T(x)$, following ideas that have been used in the case of domains
\cite{Ste,ADJMF,ADFJM},
is to use Lemma
\ref{Boggio}.  Since
\begin{equation}\label{UpsiAgmon}
\frac{H \Upsilon(x)}{\Upsilon(x)} \ge \frac{1}{\Upsilon(x)},
\end{equation}
which is positive, we can use the method of \S \ref{Agmon}
to obtain Agmon-type bounds on parts of $\Gamma_0$
that extend beyond the tunneling region. In particular,
using $f = \eta F \psi$ in Lemma \ref{Boggio} and inserting
\eqref{UpsiAgmon} into \eqref{2ndid}, we find that
\[
\left(\frac{1}{\Upsilon}- E - \left(\frac{F^\prime}{F}\right)^2 \right) \left(F \eta \psi \right)^2 -H^\prime(x)
\le \psi^2(x) \left(\eta^\prime(x)\left(\eta(x) F^2(x)\right)^\prime\right) - G^\prime(x),
\]
where $G'$ and $H'$ will integrate to 0.  This allows us to chose
\begin{equation}\label{AgmonUps}
F(x) = \exp\left(\int_{x_0}^x\sqrt{\frac{1}{\Upsilon}- E -\delta}\right)
\end{equation}
on any region where $\frac{1}{\Upsilon}- E \ge \delta$, with $\eta$ supported in the same region,
and proceed as before.
In this manner,
bounds based on Agmon's method are obtainable in
parts of $\mathcal{C}_E$ where
$V(x) < E < \frac{1}{\Upsilon(x)}$.

We next turn to the construction of a
torsion-type landscape function on a quantum graph,
considering
first the case of
a set of abutting intervals $[x_i,x_{i+1}]$,
containing no vertices.  Suppose $x_i < y_i< x_{i+1}$ and that
on this interval
\begin{equation}\label{quadlb}
V(x) \ge V_i + b_i^2 (x-y_i)^2,
\end{equation}
with $b_i \ge 0$,
$V_i \ge 0$.
We can construct a landscape function on this interval in the form of a Gaussian function plus a constant, as follows.
We temporarily set $i=1$, $y_1=0$, $b_i=b$ for simplicity.
If $b > 0$,
define
$\Upsilon_0(x) := A\left[\frac12 +\exp(-bx^2/2)\right]$.
Then
\begin{align}\label{basiclsf}
\left[-\frac{d^2}{dx^2} + V\right] \Upsilon_0&\ge
\left[  -\frac{d^2}{dx^2} + V_1 + b^2 x^2\right]\Upsilon_0\nonumber\\
&=
A \left[\frac{V_1+ b^2 x^2}{2} + (b+ V_1) e^{-bx^2/2} \right].
\end{align}
We want to assign $A$ the minimal possible value so that the right side of
\eqref{basiclsf} $\ge 1$ on the interval $[x_1,x_2]$. To do so we find the minimum of
\[f(x):=\frac{V_1+ b^2 x^2}{2} + (b+ V_1) e^{-bx^2/2}
\]
on $[x_1,x_2]$. Taking the derivative and setting it to 0 we obtain
\[
b^2x - bx(b+V_1)e^{-bx^2/2} =0
\]
Since $b + V_1 > 0$ by assumption, the minima occur at \[
x^2 = -\frac{2}{b} \ln \frac{b}{b+V_1}
\]
and the value of such a minimum
is $\frac{V_1}{2} - b \ln \frac{b}{b+V_1} + b$. When $-\frac{2}{b} \ln \frac{b}{b+V_1}> 0$ we can obtain a real value for the minimizer,  which gives
\[
\frac{1}{A} = \min\left\{\frac{V_1}{2} - b \ln \frac{b}{b+V_1} + b, f(x_1), f(x_2)\right\}.
\]


We illustrate this construction in Case Study \ref{maybebetter}. We also observe that with a slight weakening of the inequality,
an explicit value of $A$ can be assigned
using the fact that that $e^{-y} + y\ge \max\{1,y\}$ for all $y > 0$,
\emph{viz}.,
\begin{equation}\label{e:A0est}
A_0 = \frac{1}{b+\frac{V_1}2}.
\end{equation}

When $b=0$, $\Upsilon_0$ can be chosen, for example,
as an elementary quadratic
of the form $a_1 - b_1 x^2$, such that
$V_1 (a_1 - b_1 x^2) + 2 b_0 \ge 1$ on $[x_1, x_2]$.
If $V_1$ is large it may even suffice for these purposes to
choose $b_1 = 0$, i.e., $\Upsilon_0$ may be constant on $[x_1, x_2]$.
In practice, where $b_1=0$ the upper bound given by
a quadratic
$\Upsilon_0$ will often either be weaker than the Agmon estimate, when applicable,
or,
as illustrated in Case Studies
\ref{scarycasestudy} and \ref{scarycasestudy2}, no better than the uniform bound of
Proposition~\ref{linftybd1}.  It is included here only to ensure that a single, seamless landscape function
can be constructed on a set of concatenated intervals.

Letting $T(x) = \Upsilon_0(x)$ on $[x_1,x_2]$ we obtain from \eqref{ttlsfn} that
\begin{equation}\label{Ups0ub}
|\psi(x)| \leq \max\{W_+(x_1), W_+(x_2)\} + E \|\psi\|_{L^\infty([x_1, x_2])}\Upsilon_0(x).
\end{equation}
For the bound \eqref{Ups0ub}
to be nontrivial, we want $E$ to be small in comparison
with $b+\frac{V_1}{2}$
and we shall need to address the boundary values at
$x_{1,2}$.

First, however, we show how to concatenate the construction of a landscape function in a multiple-well region.  Suppose now that $V(x)$ satisfies inequalities
of the form \eqref{quadlb} on the interval $[x_1,x_2]$ with $y=y_1$, the analogous inequality on the interval $[x_2,x_3]$ with $y=y_2$, etc.
The landscape functions as constructed above will be denoted $\Upsilon_0(x;b_i,V_i,y_i)$.  They do not
{\em a priori}
define a $C^1$ function at the ends of the
intervals $x_i$, but that problem can be fixed.

\begin{itemize}
\item[$\bullet$]
Step 1.  Beginning with $\Upsilon_0(x_i;b_i,V_i,y_i)$ as defined above,
we first ensure that the derivatives are zero at the end points of its interval by
adding functions of the form
\[
\frac{\epsilon}{2}  |\Upsilon_0^\prime(x_i;b_i,V_i,y_i)| \left(1 - \frac{|x - x_i|}{\epsilon}\right)^2 \chi_{[x_i,x_i +\epsilon]}(x),
\]
resp.
\[
\frac{\epsilon}{2}  |\Upsilon_0^\prime(x_{i+1};b_i,V_i,y_i)| \left(1 - \frac{|x - x_{i+1}|}{\epsilon}\right)^2 \chi_{[x_{i+1}-\epsilon,x_{i+1}]}(x)
\]
for $\epsilon$ small enough that the supports of these functions are contained in $(x_{i-1},x_{i+1})$.
Evidently, the quantity $\epsilon$ may be chosen in some convenient and roughly optimal way, depending on the parameters $x_i, y_i, b_i, V_i$, and
need not have the same value at the two ends.  We denote the sum of these two local quadratic functions $\rho_i(x)$.

\item[
$\bullet$]
Step 2:  Add positive constants $c_i \chi_{[x_i, x_{i+1}]}$ on a subset of the intervals $[x_i, x_{i+1}]$ in order make the concatenated function continuous.
(Although we are describing here a universal way to piece together the landscape construction,
in individual cases a good alternative to adding constants is often to choose the ends of the intervals $x_i$ in advance so that
$\Upsilon_0(x_i;b_{i-1},V_{i-1}$, $y_{i-1}) = \Upsilon_0(x_i;b_i,V_i,y_i)$,
making use of the fact that $\Upsilon_0(x;b_{i-1},V_{i-1},y_{i-1})$ decreases as $x \uparrow x_i$, whereas
$\Upsilon_0(x;b_i,V_i,y_i)$ increases as $x$ increases beyond $x_i$.)

\item[
$\bullet$]
Step 3.  If necessary,
an overall constant $c_0$ is also added to ensure that $H \Upsilon \ge 1$ after Step 1 has been carried out.
\end{itemize}

The explicit expression
\begin{equation}\label{interval.lsf}
\Upsilon(x) := c_0+ \sum_i{(\Upsilon_0(x;b_i,V_i,y_i)+ c_i) \chi_{[x_i, x_{i+1}]}+\rho_i(x)}
\end{equation}
then has all the properties required of a landscape function on a sequence of abutting intervals, in the absence of vertices.

When adapting this construction to quantum graphs, in the vicinity of vertices
we use star-graphs instead of intervals.
When we overlay a subgraph with abutting ingtervals and star-graphs, we
must take into account the vertex conditions and the possibility of closed loops. For the purpose of constructing a consistent landscape function, we impose additional symmetry conditions on our star-graphs.

\begin{thm}
Any connected subset $\mathcal{S}$
of a quantum graph
can be overlaid in an algorithmic manner
with abutting intervals and star graphs
on which
a function of the form \eqref{interval.lsf} can be defined, in terms of which
$|\psi(x)| - \Upsilon(x)$
does not have a local maximum on $\text{int}(\mathcal{S})$.
\end{thm}

\begin{rem}
The construction in this theorem will be illustrated
on a small scale in Case Study \ref{Mathieu}.
An interesting situation arises when $\partial\mathcal{S} \subset \mathcal{T}_E$, because then the boundary values
can be controlled
by the Agmon estimates of the previous section.
\end{rem}

\begin{proof}
The vertices do not pose much difficulty in adapting \eqref{interval.lsf}, because the maximum principle of Lemma \ref{maxprin} applies with
super-Kirchhoff vertex conditions.  We can and shall exclude intervals for which vertices occur at the endpoints.  Each vertex
can thus be regarded as interior to a
subinterval of a pair of its edges.  If it should happen that the vertex $v$ coincides with a maximal point $y_i$ for which an estimate
\eqref{quadlb} holds on a set of subintervals of all pairs of the edges incident to $v$, then $\pm \psi - C \Upsilon_0$ satisfies the Kirchhoff conditions at
$v$, where $\Upsilon_0$ is defined as above on each of those subintervals.  In this circumstance, we can proceed as above.
Otherwise, for any given $v$ we privilege one of the adjacent edges, ${\bf e}_p$ to contain a value $y_i$ with respect to which
an inequality of the form \eqref{quadlb} holds on subintervals of ${\bf e}_p \cup {\bf e}$ uniformly for all edges
${\bf e} \ne {\bf e}_p$ incident to $v$.  We now choose the function $\Upsilon_0$ constructed above identically on each of these
subintervals of ${\bf e}_p \cup {\bf e}$.  Because $\Upsilon_0$ decreases outward from $v$ along each ${\bf e} \ne {\bf e}_p$
and only increases from $v$ along ${\bf e}_p$, any function of the form
$\pm \psi - C \Upsilon_0$ satisfies super-Kirchhoff conditions at
$v$, and the maximum principle applies.

Next we arrange that the landscape function constructed on a concatenated set of intervals and star graphs remains
$C^1$ even when the intervals compose a closed cycle.

\begin{itemize}
\item[$\bullet$]  Step 1.  First, we may restrict ourselves to
using star-graphs in the covering of $\mathcal{S}$
that a) contain no more than one vertex,
and b) are symmetric with respect to $y_i$.
Consequently, the functions $\Upsilon_0(x;b_i,V_i,y_i)$ on these star-graphs will be
symmetric in $x$ with respect to reflection through $y_i$.

\item[$\bullet$]  Step 2.  On each star-graph, on neighborhoods of its ends we add quadratic functions $\rho_i$ as defined above, to ensure that
$\Upsilon_0^\prime(x;b_i,V_i,y_i) + \rho_i^\prime(x) = 0$ when $x$ is at an endpoint.

\item[$\bullet$]  Step 3.  We subtract a constant on each star-graph so that $\Upsilon_0(x;b_i,V_i,y_i) + \rho_i(x) -c_i = 0$ at the
endpoints.  The resulting functions compose a $C^1$ function on all of $\mathcal{S}$.

\item[$\bullet$]  Step 4.  We now add a single constant $c_0$ on $\mathcal{S}$ sufficiently large to ensure that
$H \Upsilon \ge 1$ for all $x \in \mathcal{S}$.
\end{itemize}

\end{proof}

\section{Landscape functions in the high-energy r\'egime}\label{high}

A good tool for controlling high-energy eigenfunctions is
a theorem of Davies \cite{Dav83} using a differential inequality:
\begin{thm}[Davies] \label{vopbd}
Given a real-valued solution of \eqref{EVE} on an edge ${\bf e}$ and $E_m < E$, define
$$g(x,E,E_m) := (\psi(x))^2 +\frac{(\psi^\prime(x))^2}{E-E_m}.$$
Then
for $x,y \in {\bf e}$, choosing a parametrization so that $x \ge y$,
\begin{equation}\label{Davbd}
g(x,E,E_m) \le g(y,E,E_m) \exp\left(\frac{1}{\sqrt{E-E_m}} \int_y^x|V(t) - E_m| dt\right).
\end{equation}
\end{thm}

We have rewritten this result in a form compatible with our presentation and have inserted a useful parameter $E_m$
not used in \cite{Dav83}.
While this bound is universally valid on intervals, it is most striking when $E$ is large, as it implies that $g$ is
slowly varying.  The shortcoming of Theorem \ref{vopbd} is that since it involves the derivative, which is
not generally continuous on
a path that passes a vertex, it is difficult to adapt to regions containing vertices.   That this is a true difficulty is illustrated in Case Studies
\ref{scarycasestudy} and \ref{scarycasestudy2}, in which
the magnitude of an eigenfunction differs dramatically on parts of a graph separated by vertices.

For completeness we offer a proof of Theorem \ref{vopbd}.

\begin{proof}
Using the freedom to redefine $V \to V - E_m$ if simultaneously $E - E_m$, we may
set $E_m = 0$ in the proof.  We take
the derivative of $g$:
\begin{equation}\label{derivg}
g' = 2\psi \psi' + \frac{2\psi' \psi''}{E} = 2\psi \psi'\left(1 + \frac{V - E}{E}\right) = \frac{2 V}{E} \psi \psi'.
\end{equation}
This yields
\begin{equation}
g' \leq \frac{|V|}{\sqrt E}(\psi^2 + (\psi')^2/E) = \frac{|V|}{\sqrt E}g.
\end{equation}
Dividing by $g$ and integrating yields the result.
\end{proof}

In concert with Sturm oscillation theory, Theorem \ref{vopbd} can sometimes be used to obtain ``landscape functions'' that do not contain derivatives explicitly, so long as
vertices are avoided.

\begin{cor}
Let $\psi(x)$ be a real solution of \eqref{EVE} on an interval $\mathcal{I}$, and suppose that
$E - V(x) \ge k^2 > 0$ on a subinterval $\mathcal{I}_-= (x_1,x_2)$ of length at least $\frac \pi k$.
Then for any $x \ge x_1$ and any $E_m < E$,
\begin{align}\label{g-to-lnscp}
|\psi(x)| &\le \sqrt{\psi(x)^2 + \frac{\psi'(x)^2}{E-E_m}}\nonumber\\
& \le \|\psi\|_{L^\infty(\mathcal{I}_-)} \exp\left(\frac{1}{2 \sqrt{E-E_m}} \int_{x_1}^x|V(t) - E_m| dt\right).
\end{align}
The analogous statement holds for any $x \le x_2$.
\end{cor}

\begin{proof}
According to the Sturm Oscillation Theorem, in any closed interval of length  $\frac \pi k$, $\psi'(x)$ must vanish at least once, and at any such point
$g(x) = |\psi(x)| \le \|\psi\|_{L^\infty(\mathcal{I}_-)}$.  We now apply the Theorem, taking into account that the location of the maximum of
$|\psi(x)|$ in $\mathcal{I}_-$ is not specified and hence extending the range of the integral to begin at $x_1$.
\end{proof}

The usefulness of  \eqref{Davbd} is illustrated in Case Study \ref{tetrah}.

\section{Transition r\'egime estimates}\label{trans}

In the section we provide a final set of upper bounds on $|\psi(x)|$, which have advantages when $V(x) - E$ is small, which we refer to a the
{\em transition r\'egime}.  We begin with the Agmon method, but make
different choices of the
functions that appear. In particular we will choose
$F\equiv 1$ in the basic identities
\eqref{1stid.b}--\eqref{2ndid}, and choose $\eta$ to be supported in some region
where the negative part of $V - E$ is small.  We think of this set as a particular ``window''
and find that the value of $\psi(x)$ is controlled by its values around the border of the window.

\begin{thm}\label{window}
Consider a region $W \subset \Gamma$ such that for some $\ell > 0$,
$B_{\ell} := \{x \in W: \dist(x, \partial W) \le \ell\}$.  If $B_{\ell} $ contains no vertices, then for all $x \in W$ such that
$\dist(x, \partial W) \ge \ell$,
\begin{equation}\label{windowbd}
|\psi(x)|^2 \le \left(\frac{1}{\ell^2} \int_{B_{\ell}} \psi^2 + \int_W (E - V(x))_+ \psi^2\right) \dist(x, \partial W).
\end{equation}
\end{thm}
Here $z_+ := \max(z,0)$.  An obvious simple upper bound for a normalized eigenfunction is
\begin{equation}\nonumber
|\psi(x)|^2 \le \left( \frac{1}{\ell^2} + \max_W (E - V(x))_+\right) \dist(x, \partial W) .
\end{equation}
Hence, if we can choose $\ell$ large on a transition r\'egime or we have information that
$|\psi|$ is small near its boundary, e.g., because of an Agmon estimate, we are ensured that $\psi$
remains small on the window $W$.

\begin{proof}
We set $F\equiv 1$  in
\eqref{1stid.b}--\eqref{2ndid} and choose $\eta(x) = 1$ for $\dist(x, \partial W) \ge \ell$,
$\eta(x) = 0$ on $W^c$, and interpolate with a linear ramp on $B_\ell$, hence $|\eta'(x)| = \frac{1}{\ell}$ on
$B_\ell$.  By integrating \eqref{1stid.b}--\eqref{2ndid}, we find that
\begin{equation}
\int_W{\left(((\eta(x) \psi(x))')^2 + (V(x) - E) (\eta(x) \psi(x))^2\right) dx}  = \frac{1}{\ell^2} \int_{B_\ell} {\psi(x)^2 dx} .
\end{equation}
For $x \in W \setminus B_\ell$, choosing a parametrization so that $x=0$ corresponds to the nearest point on a path to
$x$ for which $\eta(0) = 0$, we can write
\begin{align*}
\psi(x)^2 &= \left(\int_0^x (\eta(t) \psi(t))' dt\right)^2\\
& \le x \int_0^x ((\eta(t) \psi(t))')^2 dt\\
& \le  \dist(x, \partial W) \int_{W} ((\eta(t) \psi(t))')^2 dt\\
& \le  \dist(x, \partial W) \left( \frac{1}{\ell^2} \int_{B_\ell} {\psi(x)^2 dx}
+ \int_W (E - V(x))+- \psi(x)^2 dx \right),
\end{align*}
as claimed.
\end{proof}

An alternative on an edge where $|V(x) - E|$ is small and no vertices are encountered,
good pointwise control of eigenfunctions can be
obtained with Gronwall's inequality.
Using the Fundamental Theorem of Calculus we write
\[
\psi(x) = \psi(x_0) + (x-x_0) \psi'(x_0) + \int_{x_0}^x{(x-t) (V(t) - E) \psi(t) dt},
\]
so
\[
|\psi(x)| \le |\psi(x_0) + (x-x_0) \psi'(x_0)| + \int_{x_0}^x{(x-t) |V(t) - E| |\psi(t)| dt},
\]
to which Gronwall's inequality as stated in \cite{Hart} applies, yielding
\begin{equation}\label{Gronw}
|\psi(x)| \le |\psi(x_0) + (x-x_0) \psi'(x_0)|  \exp\left( \int_{x_0}^x{(x-t) |V(t) - E| dt}\right).
\end{equation}
The bound \eqref{Gronw} is of a similar type to \eqref{Davbd}, one being more useful when
$E \gg V(x)$ and the other when $E \approx V(v)$.

\section{Case studies}

\begin{enumerate}
\item\label{scarycasestudy}
Our first two case studies show that a wavefunction $\psi$ can be concentrated to an arbitrarily large extent, even
completely, in subsets of $\mathcal{C}_E$, the
part of the graph where $E > V(x)$, while being small or even vanishing in other subsets of $\mathcal{C}_E$
which do not differ in any meaningful way from the sets on which $\psi$ is concentrated.
Consider a quantum graph with a constant potential $V=0$ which includes several circles
$\mathcal{C}_j$
of length $2 \pi/k$, which have been connected by edges.
We assume two edges per circle and use a coordinate system on each $\mathcal{C}_j$
so that the edges connect at $x=0, \pi$.
Letting $E = k^2$, on the $j$-th circle we can have an eigenfunction
$\mu_j \sin(k x)$, which vanishes at the nodal points $x=0, \pi$.  We suppose that the connecting edges
are attached at these nodal points and that the eigenfunction
equals $0$ on every connecting edge.  The numbers $\mu_j$ can be assigned arbitrarily, showing that there is no control
wha tsoever of the magnitude of the eigenfunction on
a given circle in terms of its values elsewhere!
We can even shrink the edges in this example so that pairs of circles are in direct contact.
\item\label{scarycasestudy2}
As a variant of the previous case study,
we show that the problem is not that the eigenfunctions can vanish.  On the same graph, let the eigenfunctions have the form
$\kappa_j \cos(k x) + \mu_j \sin(k x)$ on the circles $\mathcal{C}_j$.  We equip the connecting edges with any set of
Sturm-Liouville eigenfunctions having the eigenvalues $k^2$ and Neumann conditions at the ends.  We choose
$\kappa_j$ to guarantee continuity of the eigenfunctions at the vertices where the circles meet the edges, and observe
that the Kirchhoff conditions are satisfied at those vertices by construction.

\item
It is also possible for an eigenfunction $\psi$ to concentrate on
$\mathcal{T}_E$, the
part of the graph where $E < V(x)$, as shown by the example
of a half line with a little circle attached at the origin.
The potential is a constant on the circle, $0$ on $[0, a)$,
and some other constant on $[a, \infty)$.  The constants and the eigenvalue are chosen
so that that the eigenfunction is constant on the circle.
We can take $a$ and the size of the circle small and show explicitly that almost all of the $L^2$ norm of
$\psi$ arises from the part of $\psi$ supported in the tunneling region.
In contrast to the situation where $E > V(x)$, however, the magnitude of the
eigenfunction must be small in the interior of $\mathcal{T}_E$,
in accordance with Proposition \ref{tunneldecr}.

\item
Since a landscape function is supposed to be an upper bound for {\em any} eigenfunction of a given value of $E$, we must accept that some eigenfunctons will be
quite small in a region where the appropriate landscape function is large.  We recall the analysis of perturbed double-well models, which Simon has called
``the flea on the elephant'' in \cite{Sim85}.  As descried in that work, a Schr\"odinger operator containing a classic double-well potential with a reflection symmetry
will have a ground-state eigenfunction that is symmetric and equally concentrated near the bottoms of the two wells, and an antisymmetric eigenfunction with very nearly the same
eigenvalue, likewise equally concentrated near the bottoms of the two wells except for a difference of sign.  By making a very small perturbation that is not symmetric,
the ground state eigenfunction will be
concentrated in only one of the wells, and smaller by an exponential factor in the other well.  The next state will be concentrated in the other well, and smaller by an exponential factor in the well where the first eigenfunction resides.  Meanwhile, at the level of generality of a landscape function, the upper bounds we wish to create will be virtually the same for
the first two eigenfunctons for either the strictly symmetric potential or the slightly perturbed potential.  Mathematical details are to be found in \cite{Sim85}.

\item An oscillatory example. Let us consider $V(x) = \sin (2x)$ on $[- \pi/2, \pi/2]$ and we take periodic boundary conditions so we are effectively on a circle. The Agmon region is therefore on $[0, \pi/2]$. Using \eqref{e:calc} and the fact that we are working on $S^1$ we get that
    \begin{align*}
    \psi(x) &\leq \sqrt{\frac{(x + \pi/4)(3\pi/4 - x)}{\pi}}\;
\frac{2 \|\psi \|_{L^2([-\pi/2, 0])}}{\pi/2} \\
&\times e^{- \min\{\int_0^x \sqrt{\sin (2t)}dt, \int_x^{\pi/2}\sqrt{\sin (2t)}dt\}}
    \end{align*}
To make sense of this bound we note that
$e^{- \min\{\int_0^x \sqrt{\sin (2t)}dt, \int_x^{\pi/2}\sqrt{\sin (2t)}dt\}}$ provides exponential decay into the Agmon region. The square root prefactor $\sqrt{\frac{(x + \pi/4)(3\pi/4 - x)}{\pi}}$ is maximized at $x = \pi/4$ with a maximum value of $\sqrt \pi /2$. The remaining factor is double the averaged $L^2$ norm of $\psi$ on the landscape region.

\item\label{maybebetter}

Upper bounds based on the maximum principle may or may not include the value of $|\psi|$ at the boundary of the region on which
they apply, depending on the sign of $W$ on the boundary in expressions like \eqref{ttlsfn}.  A case study to illustrate the possible dependence on boundary values can be based on the classic square-well example.
Thus let $V(x) = M > 0$ when $|x| > 1$, and $V(x) = 0$ when $|x| \le 1$.
If as in \S \ref{landscape} we wish to bound $V(x)$ from below by a convex quadratic, the symmetric choices
would be
\[
V(x) \ge b^2 (x^2 - 1)
\]
on the interval $[-1,1]$.  This is not a positive function, but
we can fix that by adding $b^2$ to $V$ if we likewise replace
$E$ by $E+b^2$ in all subsequent formulae.

Since $V(x)$ is symmetric, the eigenfunctions are even or odd, in particular, on $[-1,1]$ they are proportional to
\[
\cos(\sqrt{E} x) \text{\,\,\,or\,\,\,} \sin(\sqrt{E} x),
\]
and they make a $C^1$ matching with a multiple of
\[
\exp(-\sqrt{M-E} x)
\]
for $|x| \ge 1$.  By a standard calculation of elementary quantum mechanics, the eigenvalues are determined by one of the following conditions
\[
\tan(\sqrt{E}) = \sqrt{\frac{M}{E} - 1}
\]
or
\[
- \cot(\sqrt{E}) = \sqrt{\frac{M}{E} - 1}.
\]

As a variant, this example can be modified to a quantum graph by replacing the interval $[-1,1]$
with $n$ copies of the interval, and imposing Kirchhoff conditions.  By again exploiting the symmetry, the eigenfunctions are as before and the eigenvalues are
determined by
\[
\tan(\sqrt{E}) \text{\,\,or\,} - \cot(\sqrt{E}) = \frac 1 n  \sqrt{\frac{M}{E} - 1}.
\]
The factor $\frac 1 n$ will make no qualitative difference.

The lowest eigenvalue will lie in the interval $(0,\frac{\pi^2}{4}$ (see Figure \ref{f:xing}) and by a choice of $M$ can take on any value in this range.

\begin{figure}[h]
\includegraphics[scale = .3]{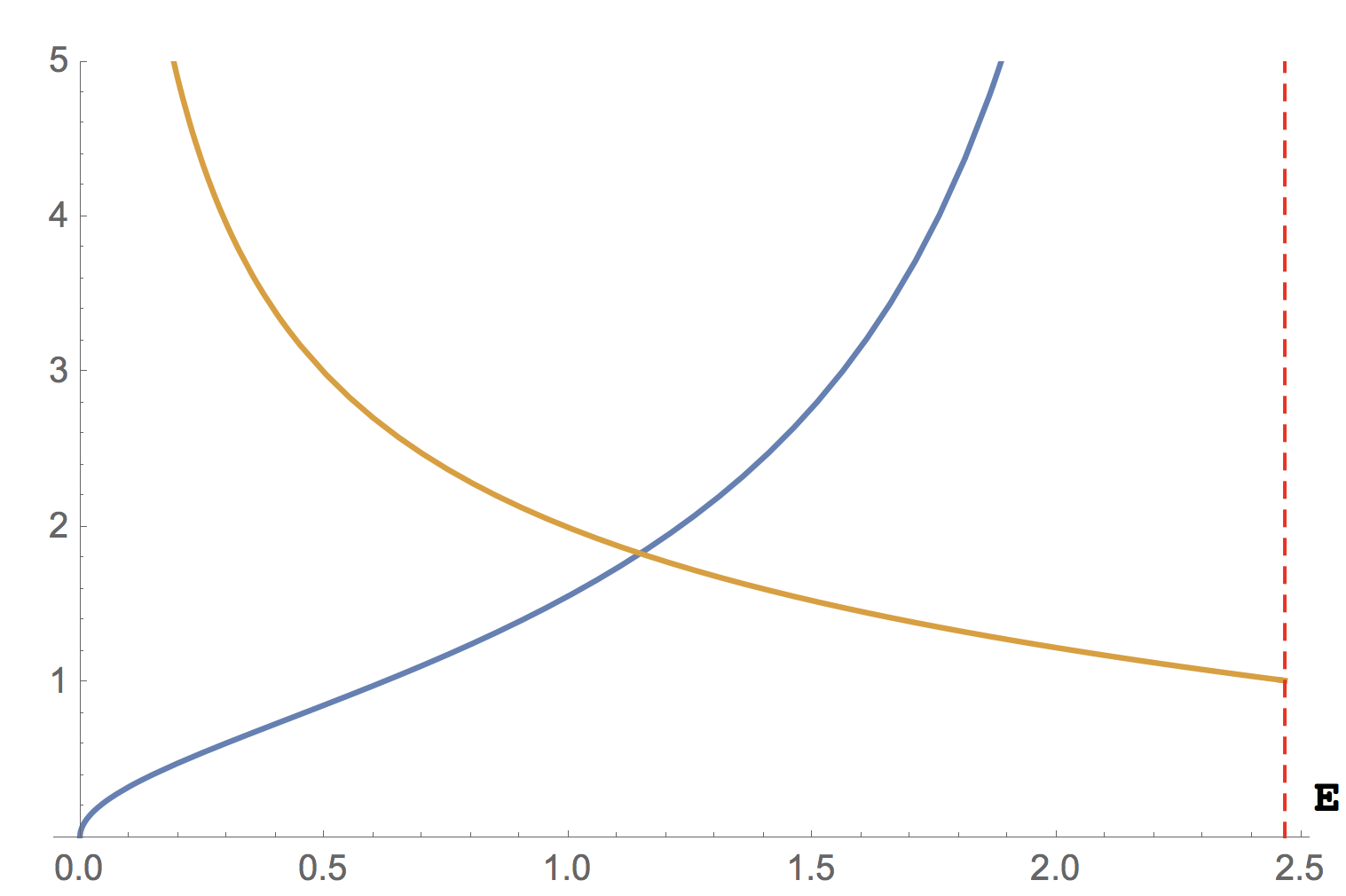}
\caption{The lowest eigenvalue in Case Study \ref{maybebetter}.}
\label{f:xing}
\end{figure}

Let us compare the corresponding eigenfunction for fixed values of $x \in [-1,1]$ with the landscape function of \S \ref{landscape}, {\em viz}.,
\[
\frac{E + b^2}{b}\left(\frac 1 2 + e^{-\frac{b x^2}{2}}\right) \|\psi\|_{L^\infty[-1,1]}.
\]
Minimizing the first factor with the choice $b = \sqrt{E}$, for a normalized eigenfunction we get
\[
|\psi(x)| - 2 \sqrt{E} \left(\frac 1 2 + e^{-\frac{\sqrt E x^2}{2}}\right) \|\psi\|_{L^\infty[-1,1]}\quad\quad\quad\quad\quad  \quad\quad\quad\quad\quad
\]
\[\quad\quad\quad\quad\quad  \quad\quad\quad\quad\quad  \le \left(|\psi(1)| - 2 \sqrt{E}\left(\frac 1 2 + e^{-\frac{\sqrt E}{2}}\right) \|\psi\|_{L^\infty[-1,1]} \right)_+
\]
for $x \in [-1,1]$.  Equivalently,
\begin{equation}
|\psi(x)| \le \max\left((\psi(1) + 2 \sqrt{E}\left(e^{-\frac{\sqrt E x^2}{2}}- e^{-\frac{\sqrt E}{2}}\right)\|\psi\|_{L^\infty[-1,1]}, \sqrt{E} \left(1 + 2e^{-\frac{\sqrt E x^2}{2}} \right) \|\psi\|_{L^\infty[-1,1]} \right).
\end{equation}
The first choice in the maximum is operative for small $E$, whereas the second is operative for larger values.  For yet larger values of $E$, however, the uniform bound of
Proposition~\ref{linftybd1}
may be superior.  The situation is depicted in Figure \ref{f:CaseStudymaybebetter}.

\begin{figure}[h]
\includegraphics[scale = .5]{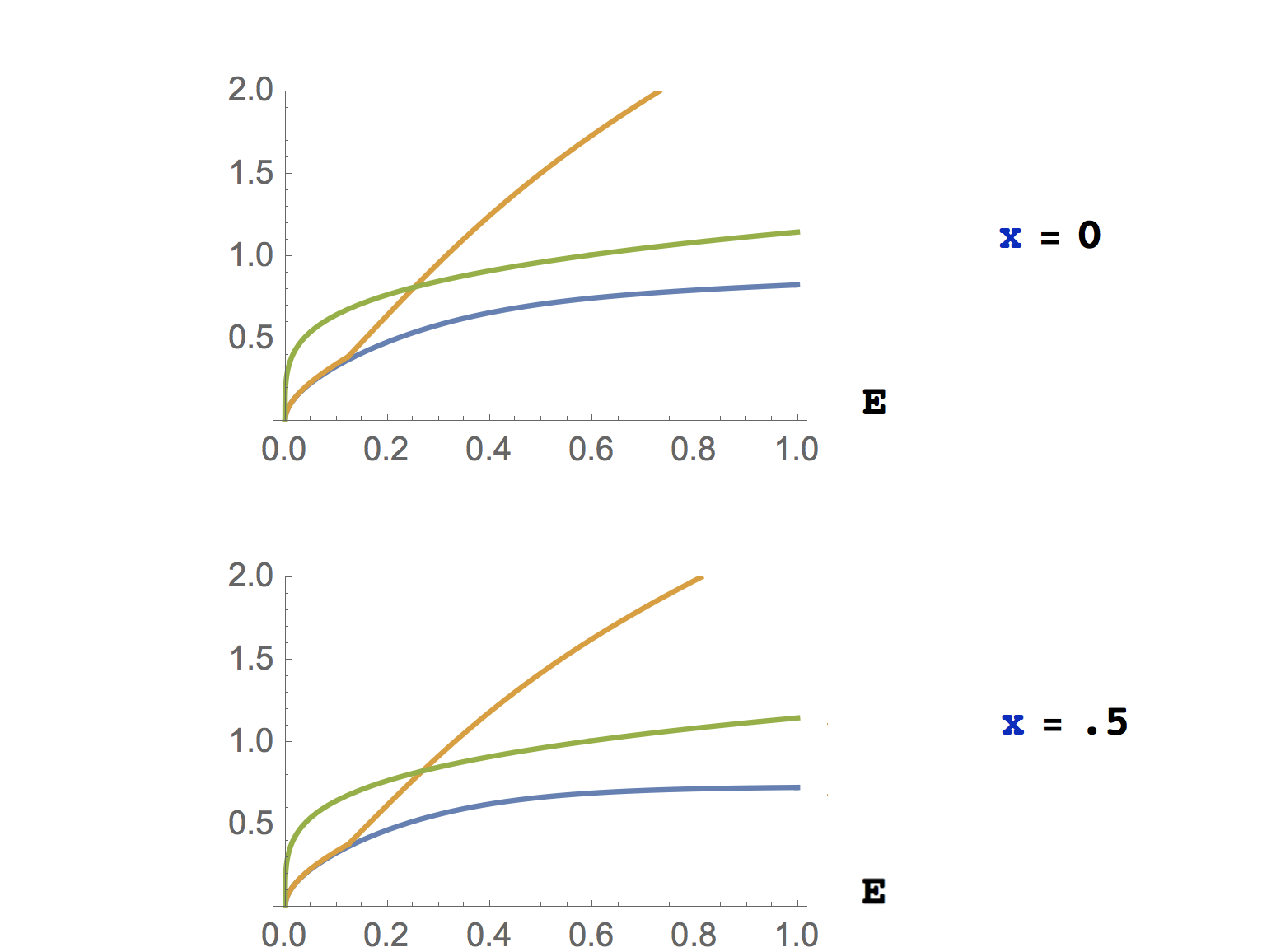}
\caption{Landscape bounds for Case Study \ref{maybebetter} with two fixed values of $x$ as functions of $E$, calculated with
{\em Mathematica}.  The eigenfunction is in blue, a torsion-type bound in gold, and the
uniform bound from Proposition~\ref{linftybd1} in green.  In illustrating the torsion-type bound we have used the maximum of the exact eigenfunction and its
value at $x=1$ (both of which can be calculated in closed form in terms of $E$), rather than approximations.}
\label{f:CaseStudymaybebetter}
\end{figure}

\item\label{Mathieu}
Mathieu functions.
Our goal in this case study is to provide evidence that the construction in \S \ref{landscape} is of interest for some range of parameter values.
The Mathieu equation in standard form is $2 \pi$-periodic, and conventionally the coefficient of the cosine potential
is denoted $2 q$.  We shift that upwards to ensure our convention of a nonnegative potential and this consider
\begin{equation}\label{MathieuEq}
- \psi^{\prime\prime} + 2 q (1+\cos(2 x)) \psi = E \psi
\end{equation}
on a circle of length $2 \pi$.
We note that the tunneling and classically allowed r\'egimes each have two connected components, and therefore construct a global
of $\Upsilon$ as in \S \ref{landscape} by concatenating truncated Gaussians and adding a constant.  We set $q=10$ and used Mathematica to calculate
an even and an odd eigenfunction with eigenvalues computed as 6.0630\dots
 and, respectively, 6.0634\dots.
In Figure \ref{f:Mathieu} the eigenfunctions are compared with
an upper bound of torsion type, using the computed $L^\infty$ norm of the normalized Mathieu eigenfunctions.  For comparison, on the intervals where $|x - n \pi| < 1$ an Agmon-type upper bound derived from Theorem \ref{tunneldecr}.  Here we incorporated the $L^2$ norm of the Mathieu eigenfunctions on intervals such as $1, \frac{\pi}{3}$, but did not attempt to optimize this interval (used as the support of our $\eta'$) or other details of the Agmon-type estimate.  Meanwhile, the uniform hypercontractive bound of Proposition~\ref{linftybd1} was computed as 1.87124, which in this case is not competitive with the other upper bounds.

Although this case study does not have vertices, since the odd Mathieu eigenfunction has zeroes at $0$ and $\pi$, we could attach an edge, or even a complicated graph, linking these two points on the circle and extend that $\psi$ by $0$, converting this into an example on a more complex graph.

\begin{figure}[h]
\includegraphics[scale = 0.05]{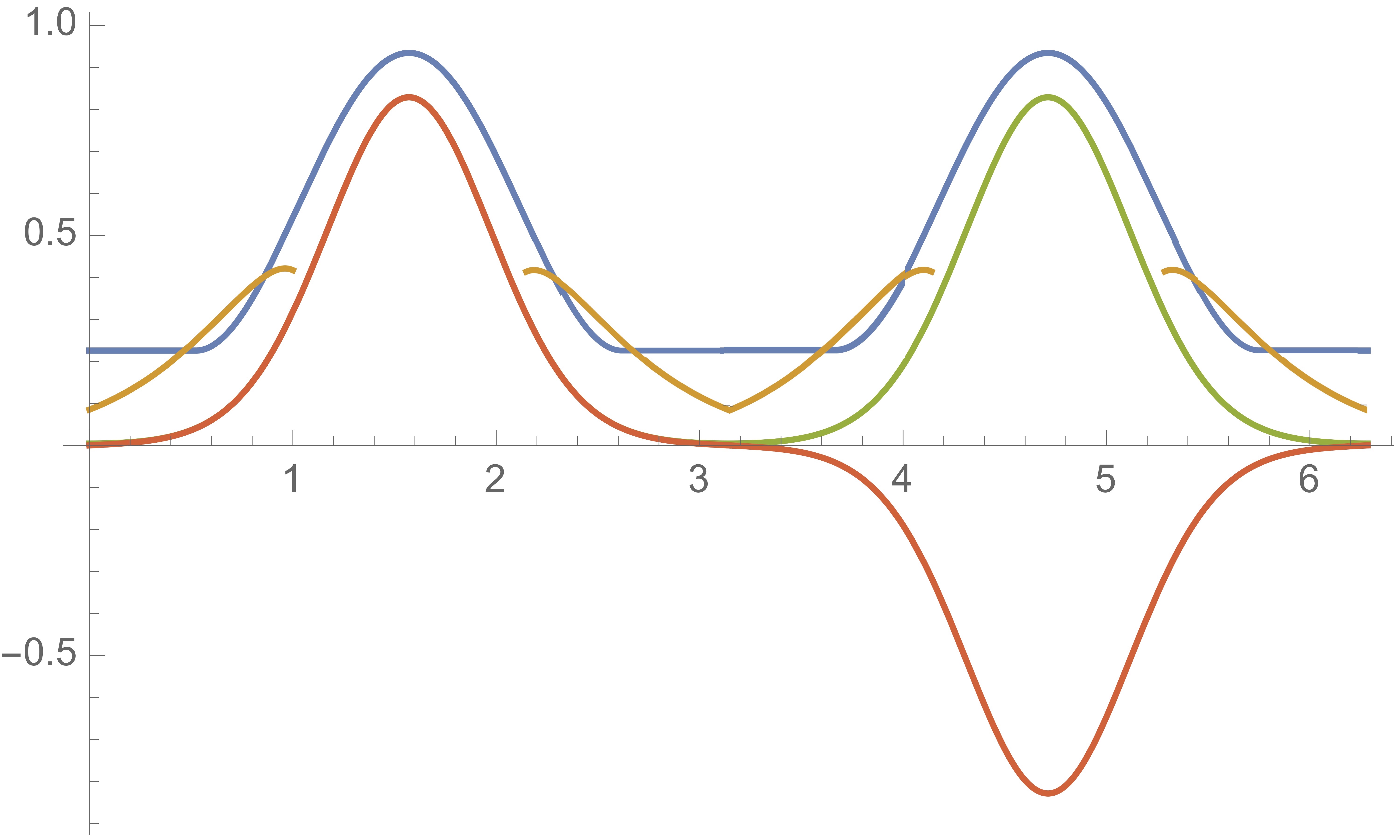}
\caption{The first two Mathieu eigenfunctions for $q=10$  (green and red), along with landscape bounds using a simplified torsion function (blue) and
Agmon's method (gold) (Case Study \ref{Mathieu}), calculated with
{\em Mathematica}.
In the torsion-type bound we have used a numerical calculation of the maximum of the Mathieu functions.  The Agmon bound is self-contained, but we have not attempted to optimize
details such as the choice of $\ell$.}
\label{f:Mathieu}
\end{figure}

\item\label{tetrah}
Our final case study shows the kind of eigenfunction control that can be achieved when $E - V(x)$ is large and an edge is long enough for the eigenfuction to oscillate many times.  In this situation the best options are the bounds of
Proposition~\ref{linftybd1} (uniform) and \ref{vopbd} (with an exponential integral).
Consider a regular tetrahedral graph with six edges of length $2 \pi$.  On three edges connected to the top vertex we will place a large, positive constant potential,
while on each of the other three edges we place a Mathieu potential of the same type as in Case Study \ref{Mathieu}, with coordinate $x=0$ at the centers of the latter edges.
Using the symmetries of the tetrahedron, we can find some explicit eigenfunctions (with some constants determined numerically), consisting of hyperbolic cosines on the edges
connecting to the top vertex and even-symmetry Mathieu functions on the other edges.  Some Mathieu parameter values for which this is possible turned out
to be $q=10$, $E=72$
and the even more highly oscillatory $q=5$, $E=300$.
As shown in
Figures \ref{f:72} and  \ref{f:300}, when $E$ is only a few times the maximum value of the potential ($72$ vs.\ $2 q = 20$), these upper bounds are of the right order of magnitude but rather crude, whereas the variable bound becomes much tighter when the ratio of $E$ to the maxiumum value of the potential  is made larger ($300$ vs. $2 q = 10$).

\begin{figure}[H]
\includegraphics[scale = 0.04]{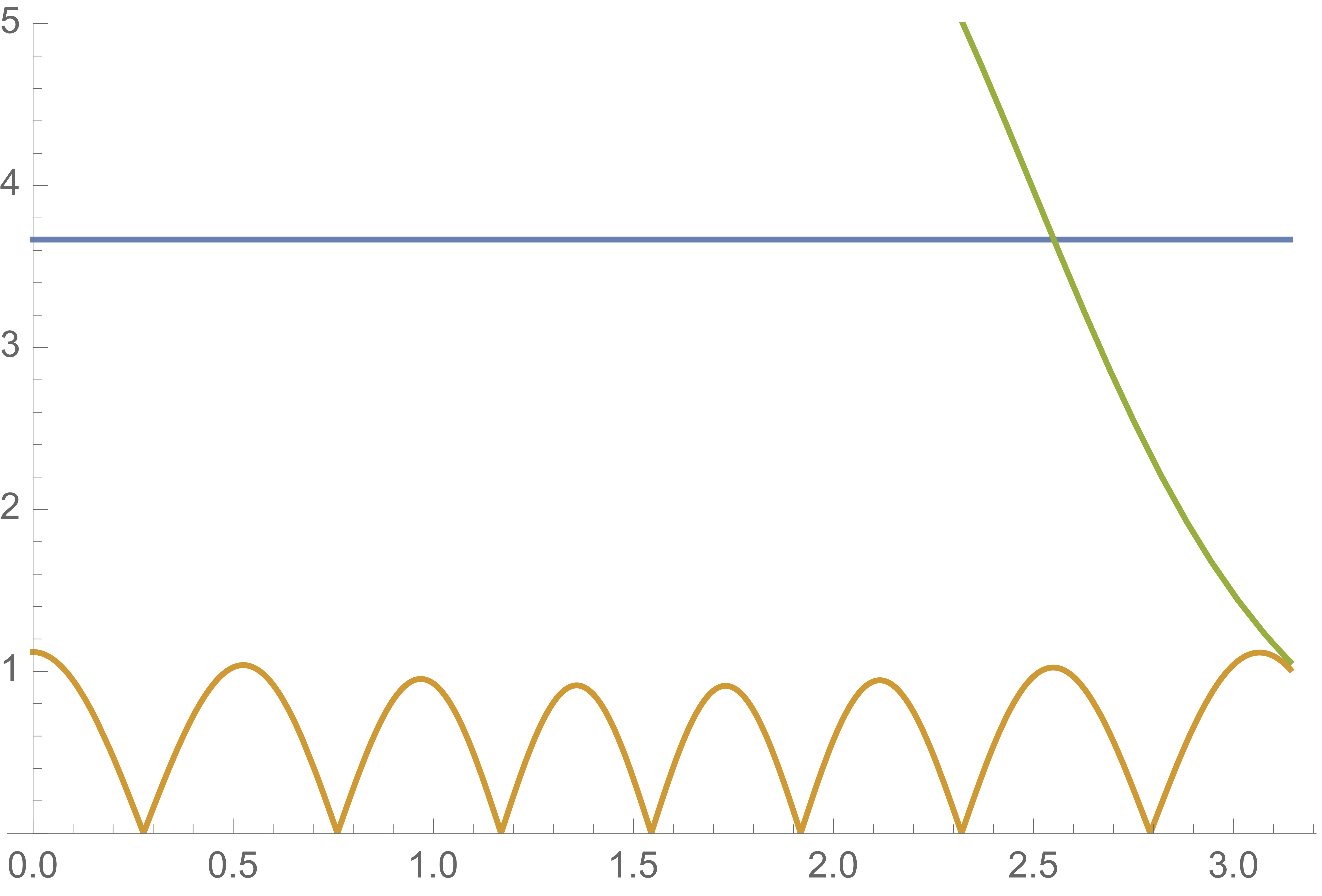}
\caption{An even Mathieu-type eigenfunction on an edge of a tetrahedron,
with $q=10$, $E=72$ (red),
shown in magnitude, along with the uniform upper bound of
Proposition~\ref{linftybd1} and the
upper bound from Theorem~\ref{vopbd} (green).
 (Case Study \ref{tetrah}).}
\label{f:72}
\end{figure}

\begin{figure}[H]
\includegraphics[scale = 0.04]{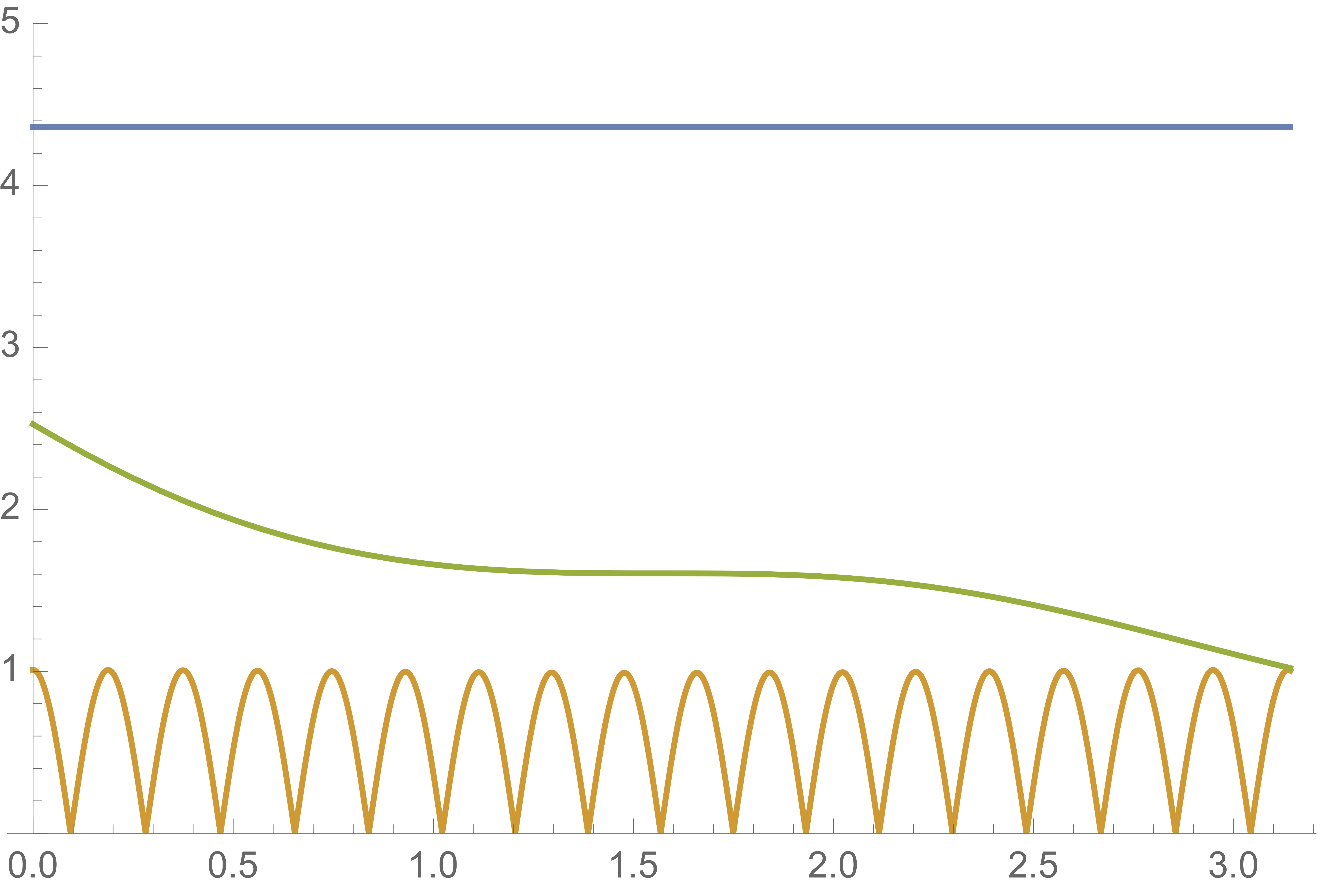}
\caption{An even Mathieu-type eigenfunction on an edge of a tetrahedron,
with $q=5$, $E=300$ (red),
shown in magnitude,
along with the uniform upper bound of
Proposition~\ref{linftybd1} and the
upper bound from Theorem~\ref{vopbd} (green).
 (Case Study \ref{tetrah}).}
\label{f:300}
\end{figure}

Figures \ref{f:72} and \ref{f:300} depict the eigenfunctions on the outer edges of the tetrahedral model,
along with the associated uniform upper bound and the upper bound of Theorem \ref{vopbd}.  Since
the eigenfunctions are even, only the interval $[0, \pi]$ is shown on the edge, which has total length $2 \pi$.
\end{enumerate}

\subsection*{Acknowledgments}

The authors are grateful to D. Jerison for correspondence about prior literature. A.M. was supported by the Royal Society [UF160569].


\end{document}